\documentclass[12pt]{article}
\usepackage{graphicx}
\usepackage{amsmath}
\usepackage{amssymb}
\usepackage{theorem}

\numberwithin{equation}{section}
   
\newtheorem{thm}{Theorem}[section]
\newtheorem{lemma}[thm]{Lemma}
\newtheorem{prop}[thm]{Proposition}
\newtheorem{cor}[thm]{Corollary}
{\theorembodyfont{\rmfamily}

\newtheorem{rmk}[thm]{Remark}
}

\newcommand{\qed}{\hfill \mbox{\raggedright \rule{.07in}{.1in}}}
 
\newenvironment{proof}{\vspace{1ex}\noindent{\bf
Proof}\hspace{0.5em}}{\hfill\qed\vspace{1ex}}
\newenvironment{pfof}[1]{\vspace{1ex}\noindent{\bf Proof of
#1}\hspace{0.5em}}{\hfill\qed\vspace{1ex}}

\newcommand{\R}{{\mathbb R}}

\newcommand{\C}{{\mathbb C}}
\newcommand{\Z}{{\mathbb Z}}

\newcommand{\eps}{{\epsilon}}

\newcommand{\Leb}{\operatorname{Leb}}
 \newcommand{\diver}{\operatorname{div}}
 \newcommand{\Int}{\operatorname{Int}}
 \newcommand{\spec}{\operatorname{spec}}
 \newcommand{\type}{\operatorname{type}}
 \newcommand{\diam}{\operatorname{diam}}
 \newcommand{\range}{\operatorname{range}}
 
\renewcommand{\Re}{\operatorname{Re}}

\newcommand{\SMALL}{\textstyle}
\newcommand{\BIG}{\displaystyle}

\newcommand{\cH}{{\mathcal{H}}}
\newcommand{\cW}{{\mathcal{W}}}
\newcommand{\cF}{{\mathcal{F}}}
\newcommand{\cC}{{\mathcal{C}}}
\newcommand{\cU}{{\mathcal{U}}}

\newcommand{\fX}{{\mathfrak{X}}}

\title{Exponential decay of correlations for nonuniformly hyperbolic flows with a $C^{1+\alpha}$ stable foliation, including the classical Lorenz attractor}

 \author{
Vitor Ara\'ujo\thanks{Departamento de Matem\'atica, Universidade Federal da Bahia, Av.\ Ademar de Barros s/n, 40170-110 Salvador, Brazil} \and  Ian Melbourne\thanks{Institute of Mathematics, University of Warwick, Coventry CV4 7AL, UK}
 }

\date{14 September 2015; Updated 28 January 2016}

\begin{document}

\maketitle

 \begin{abstract}
We prove exponential decay of correlations for a class of $C^{1+\alpha}$ uniformly hyperbolic skew product flows, subject to a uniform nonintegrability condition.
In particular, this establishes exponential decay of correlations for an open set of geometric Lorenz attractors.  As a special case, we show that the classical Lorenz attractor is robustly exponentially mixing.
 \end{abstract}

 \section{Introduction} 
 \label{sec-intro}

Although there is by now an extensive literature on statistical properties 
for large classes of flows with a certain amount of hyperbolicity, the situation for exponential decay of correlations remains poorly understood.
Groundbreaking papers by Chernov and Dolgopyat~\cite{Chernov98,Dolgopyat98a} proved exponential decay
for certain Anosov flows, namely (i) geodesic flows on compact surfaces with negative curvature, and (ii) Anosov flows with $C^1$ stable and unstable foliations.  The method was extended by~\cite{Liverani04} to cover all contact Anosov flows (which includes geodesic flows on compact negatively curved manifolds of all dimensions).

Outside the situation where there is a contact structure,~\cite{Chernov98,Dolgopyat98a} relies heavily on the smoothness of both stable and unstable foliations, a situation which is pathological~\cite{HasselblattWiklinson99}: for Anosov flows, typically neither foliation is $C^1$.

Baladi \& Vall\'ee~\cite{BaladiVallee05} introduced a method, 
extended by~\cite{AGY06}, for proving exponential decay of correlations for flows when the stable foliation is $C^2$.
This is still pathological for Anosov flows.  However, for uniformly hyperbolic (Axiom~A) flows it can happen robustly that one (but not both) of the foliations is smooth.  
This formed the basis for the paper by Ara{\'u}jo \& Varandas~\cite{AraujoVarandas12} obtaining exponential decay of correlations for a nonempty open set of geometric Lorenz attractors with $C^2$ stable foliations.  Then~\cite{ABV} 
obtained exponential decay for a nonempty open set of Axiom~A flows with $C^2$ stable foliations.

In this paper, we point out how to relax the regularity condition in~\cite{BaladiVallee05} from $C^2$ to $C^{1+\alpha}$.  Combining this with the ideas from~\cite{AGY06},
we are able to prove exponential decay of correlations for flows with a $C^{1+\alpha}$ stable foliation satisfying a uniform nonintegrability condition (UNI).  
This improvement is particularly useful in the case of the classical Lorenz attractor where the stable foliation of the flow can be shown to be $C^{1+\alpha}$
(see~\cite[Lemma~2.2]{AMV15}), but it seems unlikely that the foliation is $C^2$.
Uniform nonintegrability was recently established for a convenient induced flow for such Lorenz attractors~\cite{AMV15}.  This includes the classical Lorenz attractor and also vector fields that are $C^1$ close.  Hence we obtain:

\begin{thm} \label{thm-Lorenz}   The classical Lorenz attractor is robustly exponentially mixing.
\end{thm}

The remainder of this paper is organized as follows.
In Section~\ref{sec-semiflow}, we consider exponential decay of correlations for a class of nonuniformly expanding skew product semiflows satisfying UNI, and extend the result of~\cite{BaladiVallee05} by showing that certain $C^2$ hypotheses can be relaxed to $C^{1+\alpha}$.
In Section~\ref{sec-flow}, we prove the analogous result for nonuniformly hyperbolic skew product flows.
In Section~\ref{sec-app}, we apply our main results to 
geometric Lorenz attractors and Axiom A flows.

 \paragraph{Notation} Write $a_n=O(b_n)$ or $a_n\ll b_n$ if there is a constant $C>0$ such that $a_n\le b_n$ for all $n$.

\section{Semiflows over $C^{1+\alpha}$ expanding maps with $C^1$ roof functions}
\label{sec-semiflow}

In this section, we prove a result on exponential decay of correlations for a class of expanding semiflows satisfying a uniform nonintegrability condition
(called UNI below).
We work mainly in an abstract framework analogous to the one in~\cite{BaladiVallee05} except that we relax the condition that the expanding map is $C^2$.

\paragraph{Uniformly expanding maps}

Fix $\alpha\in(0,1]$.
Let $\{(c_m,d_m):m\ge1\}$ be a countable partition mod~$0$ of $Y=[0,1]$ and suppose
that $F:Y\to Y$ is $C^{1+\alpha}$ on each subinterval $(c_m,d_m)$ and extends to
a homeomorphism from $[c_m,d_m]$ onto $Y$.
Let $\cH=\{h:Y\to[c_m,d_m]\}$ denote the family of inverse branches of $F$,
and let $\cH_n$ denote the inverse branches for $F^n$.  

We say that $F:Y\to Y$ is a {\em $C^{1+\alpha}$ uniformly expanding map} if
there exist constants $C_1\ge1$, $\rho_0\in(0,1)$ such that 
\begin{itemize}
\item[(i)]
$|h'|_\infty\le C_1\rho_0^n$ for all $h\in\cH_n$, 
\item[(ii)]
$|\log |h'|\,|_\alpha\le C_1$ for all $h\in\cH$,  
\end{itemize}
where
$|\log |h'|\,|_\alpha=\sup_{x\neq y}|\log|h'|(x)-\log|h'|(y)|/|x-y|^\alpha$.
Under these assumptions, it is standard that there exists a unique $F$-invariant absolutely continuous probability measure $\mu$ with 
$\alpha$-H\"older density bounded above and below.

\paragraph{Expanding semiflows}
Suppose that $R:Y\to\R^+$ is $C^1$ on partition elements $(c_m,d_m)$ with
$\inf R>0$.  
Define the suspension $Y^R=\{(y,u)\in Y\times\R:0\le u\le R(y)\}/\sim$
where $(y,R(y))\sim(Fy,0)$.  The suspension flow
$F_t:Y^R\to Y^R$ is given by $F_t(y,u)=(y,u+t)$ 
computed modulo identifications, with ergodic invariant probability measure $\mu^R=(\mu\times\Leb)/\bar R$ where $\bar R=\int_Y R\,d\mu$.
We say that $F_t$ is a {\em $C^{1+\alpha}$ expanding semiflow} provided
\begin{itemize}
\item[(iii)]
$|(R\circ h)'|_\infty\le C_1$ for all $h\in\cH$.
\item[(iv)] There exists $\eps>0$ such that $\sum_{h\in\cH}e^{\eps |R\circ h|_\infty}|h'|_\infty<\infty$.
\end{itemize}

\paragraph{Uniform nonintegrability}
Let $R_n=\sum_{j=0}^{n-1}R\circ F^j$ and define 
\[
\psi_{h_1,h_2}=R_n\circ h_1-R_n\circ h_2:Y\to\R,
\]
for $h_1,h_2\in\cH_n$.
We require 
\begin{itemize}
\item[(UNI)] There exists $D>0$, and
$h_1,h_2\in\cH_{n_0}$, for some sufficiently large integer $n_0\ge1$, such that $\inf|\psi_{h_1,h_2}'|\ge D$.
\end{itemize}
The requirement ``sufficiently large'' can be made explicit.
There are constants $C_3$ and $C_4$ in Lemmas~\ref{lem-LY} and~\ref{lem-cone}
below that depend only
on $C_1$, $\rho_0$, $\alpha$ and the spectral properties of the transfer operator of $F$.  
We impose in addition the condition $C_4\ge 6C_3$.
Then we require $n_0$ sufficiently large that
\begin{align}
\label{eq-large1} & C_1^\alpha C_4\rho^{n_0}(4\pi/D)^\alpha\le {\SMALL\frac14}(2-2\cos{\SMALL\frac{\pi}{12}})^{1/2}\le {\SMALL\frac14}, \\
\label{eq-large2}  & 2\rho^{n_0}(1+C_1^\alpha C_4)\le 1, \\
\label{eq-large3}  & C_3\rho^{n_0}\le {\SMALL\frac13},
\end{align}
where $\rho=\rho_0^\alpha$.
From now on, $n_0$ and $h_1,h_2$ are fixed throughout the paper.

\paragraph{Function space}
Define $F_\alpha(Y^R)$ to consist of $L^\infty$ functions
$v:Y^R\to \R$ such that 
$\|v\|_\alpha=|v|_\infty+|v|_\alpha<\infty$ where
\[
|v|_\alpha=\sup_{(y,u)\neq(y',u)}\frac{|v(y,u)-v(y',u)|}{|y-y'|^\alpha}.
\]
Define $F_{\alpha,k}(Y^R)$ to consist of functions
with $\|v\|_{\alpha,k}=\sum_{j=0}^k \|\partial_t^jv\|_\alpha<\infty$ where $\partial_t$ denotes differentiation along the semiflow direction.

We can now state the main result in this section.
Given $v\in L^1(Y^R)$, $w\in L^\infty(Y^R)$, define the correlation function
\[
\rho_{v,w}(t)=\int v\,w\circ F_t\,d\mu^R
-\int v\,d\mu^R \int w\,d\mu^R.
\]

\begin{thm} \label{thm-semiflow}
Assume conditions (i)--(iv) and UNI.
Then there exist constants $c,C>0$ such that
\[
|\rho_{v,w}(t)|\le C e^{-c t}\|v\|_{\alpha,2}|w|_\infty,
\]
for all $v\in F_{\alpha,2}(Y^R)$, $w\in L^\infty(Y^R)$, $t>0$.
\end{thm}

An alternative, and more symmetric, formulation is to require that $v,w\in F_{\alpha,1}(Y^R)$.
The current formulation has the advantage that we can deduce the almost sure invariance principle (ASIP) for the
time-$1$ map $F_1$ of the semiflow.

\begin{cor}[ASIP] \label{cor-time1}
	Assume conditions (i)--(iv) and UNI, and suppose that $v\in F_{\alpha,3}(Y^R)$ with $\int_{Y^R}v\,d\mu^R=0$.  Then the ASIP holds for the time-$1$ map:
	passing to an enriched probability space, there exists a sequence $X_0,X_1,\ldots$ of iid normal random variables with mean zero and variance $\sigma^2$ such that
\[
	\sum_{j=0}^{n-1}v\circ F_1^j=\sum_{j=0}^{n-1}X_j
	+O(n^{1/4}(\log n)^{1/2}(\log\log n)^{1/4}),\;a.e.
\]
The variance is given by
\[
\sigma^2=\lim_{n\to\infty}\frac1n\int\Bigl(\sum_{j=0}^{n-1}v\circ F_1^j\Bigr)^2\,d\mu
=\sum_{n=-\infty}^\infty \int v\,v\circ F_1^n\,d\mu.
\]
The degenerate case $\sigma^2=0$ occurs if and only if $v=\chi\circ F_1-\chi$ for some $\chi$, where $\chi \in L^p$ for all $p<\infty$.
\end{cor}

\begin{proof}  This is immediate from~\cite[Theorem~5.2]{AMV15}.
\end{proof}

Suppose that $\phi_t:M\to M$ is an ergodic semiflow defined on a compact Riemannian manifold $M$ with probability measure $\nu$
such that there is a semiconjugacy $\pi:Y^R\to M$
satisfying $\pi_*\mu^R=\nu$ and $\phi_t\circ \pi=\pi\circ F_t$.
Suppose further that $C^2$ observables $v:M\to\R$ lift to observables
$v\circ\pi\in F_{\alpha,2}(Y^R)$.
Then it is immediate that
$\bar\rho_{v,w}(t)=\int v\,w\circ \phi_t\,d\nu
-\int v\,d\nu \int w\,d\nu$ decays exponentially for $v\in C^2(M)$,
$w\in L^\infty(M)$.   As in~\cite{Dolgopyat98a}, it follows from interpolation that $v:M \to\R$ is required only to be H\"older:

\begin{cor} \label{cor-main}
For any $\eta>0$,
there exist constants $c,C>0$ such that
\[
|\bar\rho_{v,w}(t)|\le Ce^{-c t}\|v\|_{C^\eta}|w|_\infty
\]
for all $v\in C^\eta(M)$, $w\in L^\infty(M)$, $t>0$.
\end{cor}

\begin{proof}
Let $\delta\in(0,1)$.   We can choose $\tilde v\in C^2(M)$
with $|v-\tilde v|_\infty\le\delta^\eta\|v\|_{C^\eta}$ and $\|\tilde v\|_{C^2}\le \delta^{-2}|v|_\infty$.

Let $\tilde c,\tilde C>0$ be the constants
in Theorem~\ref{thm-semiflow}.  Then
\[
|\bar\rho_{\tilde v,w}(t)|\le 
\tilde Ce^{-\tilde  c t}\|\tilde v\|_{C^2}|w|_\infty
\le \tilde Ce^{-\tilde c t}\delta^{-2}|v|_\infty|w|_\infty.
\]
Also
\[ |\bar\rho_{v,w}(t)-\bar\rho_{\tilde v,w}(t)|\le 2|v-\tilde v|_\infty|w|_\infty
\le 2\delta^\eta\|v\|_{C^\eta}|w|_\infty.
\]
Setting $\delta= e^{-\tilde c t/(2+\eta)}$, we obtain the 
desired result with $c=\tilde c\eta/(2+\eta)$.
\end{proof}

\begin{rmk} \label{rmk-ASIP}
	In this setting, we obtain from~\cite{AMV15} that the ASIP for the time-$1$ map
	$\phi_1$ holds for all mean zero observables
	$v\in C^{1+\eta}(M)$.
	Moreover, by~\cite[Section~6]{AMV15}, the degenerate case $\sigma^2=0$ is
	of infinite codimension.
\end{rmk}

The remainder of this section is devoted to the proof of Theorem~\ref{thm-semiflow}.

\subsection{Twisted transfer operators}
\label{sec-P}

For $s\in\C$, let $P_s$ denote the (non-normalised) twisted transfer operator, 
\[
P_s=\sum_{h\in\cH}A_{s,h}, \quad A_{s,h}v=e^{-sR\circ h}|h'|v\circ h.
\]

For $v:Y\to\C$, define $\|v\|_\alpha=\max\{|v|_\infty,|v|_\alpha\}$ where $|v|_\alpha=\sup_{x\neq y}|v(x)-v(y)|/|x-y|^\alpha$.
Let $C^\alpha(Y)$ denote the space of functions $v:Y\to\C$ with $\|v\|_\alpha<\infty$.
It is convenient to introduce the family of equivalent norms
\[
\|v\|_b=\max\{|v|_\infty, |v|_\alpha/(1+|b|^\alpha)\}, \quad b\in\R.
\]
Note that 
\begin{align} \label{eq-alg}
\|vw\|_b\le 2\|v\|_b\|w\|_b \quad\text{ for all $v,w\in C^\alpha(Y)$.}
\end{align}

\begin{prop} \label{prop-P}
Write $s=\sigma+ib$.
There exists $\eps\in(0,1)$ such that
the family $s\mapsto P_s$ of operators on $C^\alpha(Y)$ is continuous on
$\{\sigma>-\eps\}$.
Moreover, $\sup_{|\sigma|<\eps}\|P_s\|_b<\infty$.
\end{prop}

\begin{proof}
Since $\diam h(Y)=|h(1)-h(0)|$,
it follows from (ii) and the mean value theorem that 
$|h'|\le e^{C_1}\diam h(Y)\le e^{C_1}$.
Using the inequality $t\le 2\log(1+t)$ valid for $t\in[0,1]$, we obtain
$|h'x-h'y|/|h'y|\le 2\log(h'x/h'y)\le 2C_1|x-y|^\alpha$ and so
\begin{align} \label{eq-h'0}
|h'x-h'y|\le 2C_1|h'y||x-y|^\alpha, \quad\text{for all
$h\in\cH$, $x,y\in Y$.}
\end{align}

Note that 
\[
|A_{s,h}v|\le e^{\eps R\circ h}|h'||v|_\infty,
\]
so $\sup_{\Re s\ge-\eps}|P_s|_\infty<\infty$ by~(iv).
Also, 
\begin{align*}
& (A_{s,h}v)(x)-
 (A_{s,h}v)(y)
 =(e^{-\sigma R\circ h(x)}-e^{-\sigma R\circ h(y)})e^{-ib R\circ h(x)}|h'x|v(hx)
 \\ & \qquad +e^{-\sigma R\circ h(y)}(e^{-ibR\circ h(x)}-e^{-ibR\circ h(y)})|h'x|v(hx)
\\ &  \qquad +e^{-sR\circ h(y)}(|h'x|-|h'y|)v(hx)
+e^{-sR\circ h(y)}|h'y|(v(hx)-v(hy)),
\end{align*}
and so
\begin{align*}
|(A_{s,h}v)|_\alpha
 & \le e^{\eps|R\circ h|_\infty}|\sigma| C_1|h'|_\infty|v|_\infty
 + e^{\eps|R\circ h|_\infty}2|b|^\alpha C_1^\alpha|h'|_\infty|v|_\infty
\\ & \qquad \qquad +e^{\eps |R\circ h|_\infty}2C_1|h'|_\infty|v|_\infty
+e^{\eps |R\circ h|_\infty}|h'|_\infty|v|_\alpha C_1^\alpha \\
& \le C_1e^{\eps |R\circ h|_\infty}|h'|_\infty\{(2+|\sigma|+2|b|^\alpha)|v|_\infty
+ |v|_\alpha\},
\end{align*}
where we have used~\eqref{eq-h'0} for the third term and
 the inequality $|e^{it}-1|\le 2\min\{1,|t|\}\le 2|t|^\alpha$ for the second term.
Altogether, $\|A_{s,h}\|_b\ll (1+|\sigma|+|b|^\alpha)(1+|b|^\alpha)^{-1}e^{\eps|R\circ h|_\infty}|h'|_\infty$.
Shrinking $\epsilon$ slightly, it follows from~(iv) that the series $\sum_{h\in\cH}\|A_{s,h}\|_b$ converges 
uniformly in $\sigma\in S$ for any compact subset $S\subset[-\eps,\infty)$.~
\end{proof}

The unperturbed operator $P_0$ has a simple leading eigenvalue $\lambda_0=1$
with strictly positive $C^\alpha$ eigenfunction $f_0$.   By Proposition~\ref{prop-P}, there exists $\eps\in(0,1)$ such that $P_\sigma$ has a continuous family of simple eigenvalues $\lambda_\sigma$
for $|\sigma|<\eps$ with associated $C^\alpha$ eigenfunctions $f_\sigma$.
Shrinking $\eps$ if necessary, we can ensure that $\lambda_\sigma>0$ and
$f_\sigma$ is strictly positive for $|\sigma|<\eps$.

\begin{rmk} \label{rmk-P}
By standard perturbation theory, for any $\delta>0$ there exists
$\eps\in(0,1)$ such that
$\sup_{|\sigma|<\eps}|\lambda_\sigma-1|<\delta$,
$\sup_{|\sigma|<\eps}|f_\sigma/f_0-1|_\infty<\delta$ and
$\sup_{|\sigma|<\eps}|f_\sigma/f_0-1|_\alpha<\delta$.

Hence, we may suppose throughout that
\[
{\SMALL\frac12}\le \lambda_\sigma\le 2, \quad
{\SMALL\frac12}f_0\le f_\sigma\le 2f_0, \quad
{\SMALL\frac12}|f_0|_\alpha\le
|f_\sigma|_\alpha\le 2|f_0|_\alpha.
\]
\end{rmk}

Next, for $s=\sigma+ib$ with $|\sigma|\le\eps$ we define the normalised 
transfer operators 
\[
L_sv=(\lambda_\sigma f_\sigma)^{-1}P_s(f_\sigma v)
=(\lambda_\sigma f_\sigma)^{-1}\sum_{h\in\cH}A_{s,h}(f_\sigma v).
\]
In particular, $L_\sigma 1=1$ for all $\sigma$
and $|L_s|_\infty\le1$ for all $s$ (where defined).

\subsection{Lasota-Yorke inequality}
\label{sec-LY}

Set $C_2=C_1^2/(1-\rho)$, $\rho=\rho_0^\alpha$.  Then 
\begin{itemize}
\item[(ii$_1$)]
$|\log |h'|\,|_\alpha\le C_2$ for all $h\in\cH_n$, $n\ge1$,  
\item[(iii$_1$)] 
$|(R_n\circ h)'|_\infty\le C_2$ for all $h\in\cH_n$, $n\ge1$.
\end{itemize}
Using the arguments from the beginning of the proof of Proposition~\ref{prop-P},
it follows from (ii$_1$) that 
\begin{align} \label{eq-diam}
e^{-C_2}\diam h(Y)\le 
|h'|\le e^{C_2}\diam h(Y),\quad
\text{for all $h\in\cH_n$, $n\ge1$}.
\end{align}
 In particular
$\sum_{h\in \cH_n}|h'|\le e^{C_2}$ .
Also 
\begin{align} \label{eq-h'}
|h'x-h'y|\le 2C_2|h'y||x-y|^\alpha, \quad\text{for all
$h\in\cH_n$, $n\ge1$, $x,y\in Y$.}
\end{align}

Write
\[
L_s^nv=\lambda_\sigma^{-n} f_\sigma^{-1}\sum_{h\in\cH_n}A_{s,h,n}(f_\sigma v), \quad
A_{s,h,n}v=e^{-sR_n\circ h}|h'|v\circ h.
\]

\begin{lemma} \label{lem-LY}
There is a constant $C_3>1$ such that 
\[
|L_s^nv|_\alpha\le C_3(1+|b|^\alpha) |v|_\infty+C_3\rho^n |v|_\alpha
\le C_3(1+|b|^\alpha)\{ |v|_\infty+\rho^n \|v\|_b\},
\]
for all $s=\sigma+ib$, $|\sigma|<\eps$, and all $n\ge1$, $v\in C^\alpha(Y)$.
\end{lemma}

\begin{proof}
Compute that
\begin{align*}
& (A_{s,h,n}v)(x)- (A_{s,h,n}v)(y)  =
(e^{-\sigma R_n\circ h(x)}- e^{-\sigma R_n\circ h(y)})
e^{-ib R_n\circ h(x)}|h'x|v(hx)
\\ & \qquad +e^{-\sigma R_n\circ h(y)} (e^{-ib R_n\circ h(x)}- e^{-ib R_n\circ h(y)}) |h'x|v(hx)
\\ & \qquad +e^{-s R_n\circ h(y)} (|h'x|-|h'y|)v(hx)
+e^{-s R_n\circ h(y)} |h'y|(v(hx)-v(hy)) \\ & \quad = J_1+J_2+J_3+J_4.
\end{align*}
Using~(iii$_1$) and~\eqref{eq-h'},
\begin{align*}
|J_1| & \le 
e^{-\sigma R_n\circ h(y)}|\sigma|C_2|x-y||h'x||v(hx)|
 \le C_2(1+2C_2)e^{-\sigma R_n\circ h(y)}|h'y||v|_\infty |x-y|^\alpha
\\ & =C_2(1+2C_2)(A_{\sigma,h,n}|v|_\infty)(y) |x-y|^\alpha.
\end{align*}
Similarly,
$|J_2|\le 2C_2^\alpha(1+2C_2)|b|^\alpha (A_{\sigma,h,n}|v|_\infty)(y)|x-y|^\alpha$, 
$|J_3|\le 2C_2(A_{\sigma,h,n}|v|_\infty)(y)|x-y|^\alpha$ and
$|J_4|\le C_1^\alpha\rho^n (A_{\sigma,h,n}|v|_\alpha)(y)|x-y|^\alpha$.
Hence
\begin{align*}
 &  |(A_{s,h,n}v)(x)- (A_{s,h,n}v)(y)| 
 \\ & \qquad \qquad\qquad \le
C_2'\bigl\{(1+|b|^\alpha)(A_{\sigma,h,n}|v|_\infty)(y)
+\rho^n (A_{\sigma,h,n}|v|_\alpha)(y)\bigr\}|x-y|^\alpha.
\end{align*}

Using $|f_\sigma|_\infty|f_\sigma^{-1}|_\infty
\le 4|f_0|_\infty|f_0^{-1}|_\infty<\infty$
and $|f_\sigma|_\alpha|f_\sigma^{-1}|_\infty
\le 4|f_0|_\alpha|f_0^{-1}|_\infty<\infty$, we obtain
\begin{align*}
&  |(A_{s,h,n}(f_\sigma v))(x)- (A_{s,h,n}(f_\sigma v))(y)| \\
&\qquad \qquad\qquad  \le
C_2''\bigl\{(1+|b|^\alpha)(A_{\sigma,h,n}(f_\sigma|v|_\infty))(y)
+\rho^n (A_{\sigma,h,n}(f_\sigma |v|_\alpha))(y)\bigr\}|x-y|^\alpha.
\end{align*}
Hence
\begin{align*}
& \lambda_\sigma^{-n} f_\sigma(y)^{-1}
\Bigl|\Bigl(\sum_{h\in\cH_n}A_{s,h,n}(f_\sigma v)\Bigr)(x)-\Bigl(\sum_{h\in\cH_n} A_{s,h,n}(f_\sigma v)\Bigr)(y)\Bigr| 
\\ & \qquad \qquad\le 
C_2''\bigl\{(1+|b|^\alpha)(L_\sigma^n(|v|_\infty))(y)
+\rho^n (L_\sigma^n(|v|_\alpha))(y)\bigr\}|x-y|^\alpha
\\ & \qquad \qquad =C_2''\bigl\{(1+|b|^\alpha)|v|_\infty
+\rho^n |v|_\alpha\bigr\}|x-y|^\alpha.
\end{align*}

Finally,
\begin{align*}
(L_s^nv)(x)-(L_s^nv)(y) & =  (1-f_\sigma(x)f_\sigma(y)^{-1})\lambda_\sigma^{-n}f_\sigma(x)^{-1}\Bigl(\sum_{h\in\cH_n}A_{s,h,n}(f_\sigma v)\Bigr)(x)
\\ & \, +
 \lambda_\sigma^{-n} f_\sigma(y)^{-1}
\Bigl\{\Bigl(\sum_{h\in\cH_n}A_{s,h,n}(f_\sigma v)\Bigr)(x)-\Bigl(\sum_{h\in\cH_n} A_{s,h,n}(f_\sigma v)\Bigr)(y)\Bigr\},
\end{align*}
and so using $|1-f_\sigma(x)f_\sigma(y)^{-1}|\le C_2'''|x-y|^\alpha$,
\begin{align*}
|(L_s^nv)(x)-(L_s^nv)(y)| & \le C_2'''|x-y|^\alpha|L_\sigma^n(|v|)|_\infty
+
C_2''\{(1+|b|^\alpha)|v|_\infty
+\rho^n |v|_\alpha\}|x-y|^\alpha
\\ & \le C_3\{(1+|b|^\alpha)|v|_\infty
+\rho^n |v|_\alpha\}|x-y|^\alpha
\end{align*}
completing the proof.
\end{proof}

\begin{cor} \label{cor-LY}
$\|L_s^n\|_b\le 2C_3$ for all
$s=\sigma+ib$, $|\sigma|<\eps$, and all $n\ge1$.
\end{cor}

\begin{proof}
It is immediate that $|L_s^nv|_\infty\le|v|_\infty\le \|v\|_b$. 
By Lemma~\ref{lem-LY},
$|L_s^nv|_\alpha \le 2C_3(1+|b|^\alpha)\|v\|_b$.
Hence $\|L_s^n\|_b\le \max\{1,2C_3\}=2C_3$.
\end{proof}

\subsection{Cancellation Lemma}
\label{sec-cancel}

Suppose that $\eps\in(0,1)$ is
chosen as in Subsection~\ref{sec-P}.
Let $C_4$ be the constant in~\eqref{eq-large1} which will be specified later (see Lemma~\ref{lem-cone}).
Throughout $B_\delta(y)=\{x\in Y:|x-y|<\delta\}$.

Given $b\in\R$, we define the cone
\begin{align*}
\cC_b=\Bigl\{ & \;(u,v): u,v\in C^\alpha(Y),\;u>0,\;0\le |v|\le u,\;
|\log u|_\alpha\le C_4|b|^\alpha,\\
& \qquad\qquad |v(x)-v(y)|\le C_4|b|^\alpha u(y)|x-y|^\alpha\quad\text{for
all $x,y\in Y$} \;\Bigr\}.
\end{align*}
Let $\eta_0=\frac12(\sqrt 7 -1)\in(\frac23,1)$.

\begin{lemma} \label{lem-cancel}
Assume that the (UNI) condition is satisfied (with associated constants $D>0$ and $n_0\ge1$).   
Let $h_1,h_2\in \cH_{n_0}$ be the branches from (UNI).  

There exists $\delta>0$ and $\Delta=2\pi/D$ such that for all 
$s=\sigma+ib$, $|\sigma|<\eps$, $|b|>4\pi/D$,
and all $(u,v)\in\cC_b$ we have the following:

For every $y_0\in Y$ there exists $y_1\in B_{\Delta/|b|}(y_0)$ such that
one of the following inequalities holds
on $B_{\delta/|b|}(y_1)$:
\begin{description}
\item[Case $h_1$:]
$|A_{s,h_1,n_0}(f_\sigma v)+ A_{s,h_2,n_0}(f_\sigma v)| \le
\eta_0 A_{\sigma,h_1,n_0}(f_\sigma u)+
A_{\sigma,h_2,n_0}(f_\sigma u)$, 
\item[Case $h_2$:]
$|A_{s,h_1,n_0}(f_\sigma v)+ A_{s,h_2,n_0}(f_\sigma v)| \le
A_{\sigma,h_1,n_0}(f_\sigma u)+
\eta_0 A_{\sigma,h_2,n_0}(f_\sigma u)$.
\end{description}
\end{lemma}

\begin{proof}
Choose $\delta>0$ sufficiently small that
\[
C_1^\alpha C_4\delta^\alpha<{\SMALL\frac16}, \quad
{\SMALL\frac23} e^{C_1^\alpha C_4\delta^\alpha}<\eta_0,
\quad \delta<{\SMALL\frac{2\pi}{D}}, \quad 2C_2\delta<{\SMALL\frac{\pi}{6}}.
\]

By~(i), if $y\in B_{\delta/|b|}(y_0)$, then
$|h_my-h_my_0|\le C_1\rho_0^{n_0}\delta/|b|$ for $m=1,2$.
Hence
\begin{align} \label{eq-v1}
 |v(h_my)-v(h_my_0)|
  & \le C_4|b|^\alpha|u(h_my_0)||h_my-h_my_0|^\alpha
\\ & 
 \le C_1^\alpha C_4\delta^\alpha u(h_my_0) 
 \le {\SMALL\frac16}u(h_my_0), \quad m=1,2.
\nonumber
\end{align}
Also, for $y\in B_{\delta/|b|}(y_0)$, 
\[
|\log u(h_my_0)-\log u(h_my)|\le C_4|b|^\alpha|h_my_0-h_my|^\alpha\le C_1^\alpha C_4\delta^\alpha,
\]
and so
\begin{align} \label{eq-u}
{\SMALL \frac23} u(h_my_0)\le 
{\SMALL \frac23} e^{C_1^\alpha C_4\delta^\alpha}u(h_my)\le \eta_0u(h_my).
\end{align}
Similarly, given $\xi\in(0,4\pi/D)$, using~\eqref{eq-large1} we have that
for all $y\in B_{\xi/|b|}(y_0)$,
\begin{align} 
\label{eq-v2}
 |v(h_my)-v(h_my_0)|
  & \le C_1^\alpha C_4\rho^{n_0}(4\pi/D)^\alpha u(h_my_0) \\
  & \le {\SMALL\frac14}(2-2\cos{\SMALL\frac{\pi}{12}})^{1/2}u(h_my_0)
   \le {\SMALL\frac14}u(h_my_0), \quad m=1,2.
\nonumber
\end{align}

\noindent{\bf Case 1.}
Suppose that
$|v(h_my_0)|\le \frac12 u(h_my_0)$ for $m=1$ or $m=2$.
Then for $y\in B_{\delta/|b|}(y_0)$, using~\eqref{eq-v1} and~\eqref{eq-u},
\begin{align*}
|v(h_my)| & \le |v(h_my_0)|+|v(h_my)-v(h_my_0)|
\\ & \le {\SMALL\frac12} u(h_my_0)+{\SMALL\frac16}u(h_my_0)
= {\SMALL\frac23} u(h_my_0)
\le \eta_0 u(h_my).
\end{align*}
Hence $|A_{s,h_m,n_0}(f_\sigma v)(y)|\le
\eta_0 A_{\sigma,h_m,n_0}(f_\sigma u)(y)$ for all $y\in B_{\delta/|b|}(y_0)$
and Case~$h_m$ holds with $y_1=y_0$.

\vspace{1ex}
\noindent{\bf Case 2.}
It remains to consider the situation where
$|v(h_my_0)|> \frac12 u(h_my_0)$ for both $m=1$ and $m=2$.

Write $A_{s,h_m,n_0}(f_\sigma v)(y)=r_m(y)e^{i\theta_m(y)}$ for
$m=1,2$ and let $\theta(y)=\theta_1(y)-\theta_2(y)$.
Choose $\delta>0$ as above and $\Delta=2\pi/D$.
An elementary calculation~\cite[Lemma~2.3]{BaladiVallee05} shows that
if $\cos\theta\le\frac12$ then
$r_1e^{i\theta_1} +r_2e^{i\theta_2}\le \max\{\eta_0 r_1+r_2,
r_1+\eta_0 r_2\}$ and we are finished.
So it remains to show that $\cos\theta(y)\le\frac12$ for
all $y\in B_{\delta/|b|}(y_1)$ for some $y_1\in B_{\Delta/|b|}(y_0)$.
Equivalently, we must show that $|\theta(y)-\pi|\le 2\pi/3$.
Throughout, it suffices to restrict to
$y\in B_{\xi/|b|}(y_0)$
where $\xi= \delta+\Delta<2\Delta=4\pi/D$.

Note that $\theta=V-b\psi$ where $\psi=\psi_{h_1,h_2}$ and
$V=\arg(v\circ h_1)-\arg(v\circ h_2)$.
We begin by estimating $V(y)-V(y_0)$ for
$y\in B_{\xi/|b|}(y_0)$.
For this, it is useful to note by basic trigonometry that if $|z_1|,|z_2|\ge c$ and $|z_1-z_2|\le c(2-2\cos\omega)^{1/2}$ where $c>0$ and $|\omega|<\pi$,
then $|\arg(z_1)-\arg(z_2)|\le\omega$.
For $m=1,2$,
\begin{align} \label{eq-arg1}
|v(h_my)-v(h_my_0)|
\le {\SMALL\frac14}u(h_my_0)(2-2\cos{\SMALL\frac{\pi}{12}})^{1/2},
\end{align}
by~\eqref{eq-v2}.
Using in addition that we are in Case~2,
\begin{align} \label{eq-arg2}
\nonumber
|v(h_my)| & \ge |v(h_my_0)|-|v(h_my_0)-v(h_my)|
\\ & \ge {\SMALL\frac12}u(h_my_0)- 
 {\SMALL\frac14}u(h_my_0)= 
 {\SMALL\frac14}u(h_my_0).  
\end{align}
It follows from~\eqref{eq-arg1} and~\eqref{eq-arg2} that
$|\arg(v(h_my))-\arg(v(h_my_0))|\le \pi/12$.
We conclude that
\begin{align} \label{eq-V}
|V(y)-V(y_0)|\le \pi/6.
\end{align}

By (UNI),
\[
|b(\psi(z)-\psi(y_0))|\ge |b||z-y_0|\inf|\psi'|\ge D|b||z-y_0|=(2\pi/\Delta)|b||z-y_0|.
\]
Since $|b|>4\pi/D$, the interval
$B_{\Delta/|b|}(y_0)\subset Y$ contains an interval of length at least $\Delta/|b|$,
so it follows that $b(\psi(z)-\psi(y_0))$ fills out an interval around $0$ of length at least $2\pi$ as $z$ varies in $B_{\Delta/|b|}(y_0)$.
In particular, we can choose $y_1\in B_{\Delta/|b|}(y_0)$ such that
\[
b(\psi(y_1)-\psi(y_0))=\theta(y_0)-\pi\, \bmod 2\pi.
\]
Hence
\[
\theta(y_1)-\pi=V(y_1)-b\psi(y_1)-\pi+\theta(y_0)-V(y_0)+b\psi(y_0)=V(y_1)-V(y_0),
\]
so by~\eqref{eq-V},
$|\theta(y_1)-\pi|\le \pi/6$.
It follows from (iii$_1$) that $|\psi'|_\infty\le 2C_2$.
Hence for $y\in B_{\delta/|b|}(y_1)$,
\begin{align*}
|\theta(y)-\pi| & \le \pi/6+|\theta(y)-\theta(y_1)|
\\ & \le \pi/6+|b||\psi(y)-\psi(y_1)|+|V(y)-V(y_0)|+|V(y_1)-V(y_0)|
\\ & \le \pi/6+2C_2\delta+\pi/6+\pi/6\le 2\pi/3,
\end{align*}
as required.
\end{proof}

For each choice of $y_0$ in Lemma~\ref{lem-cancel} we let $I$ denote a closed
interval containing $B_{\delta/|b|}(y_1)$ on which the conclusion of the lemma holds.
Write $\type(I)=h_m$ if we are in case $h_m$.
Then we can find finitely many disjoint intervals $I_j=[a_j,b_{j+1}]$, $j=0,\dots,N-1$, (where $0=b_0\le a_0<b_1<a_1<\dots<b_N\le a_N=1$)
of $\type(I_j)\in\{h_1,h_2\}$ with $\diam(I_j)\in[\delta/|b|,2\delta/|b|]$ and gaps
$J_j=[b_j,a_j]$, $j=0,\dots,N$ with $0<\diam(J_j)\le 2\Delta/|b|$.

Let $\eta\in[\eta_0,1)$ and define $\chi:Y\to[\eta,1]$ as follows:
\begin{itemize}
\item Set $\chi\equiv1$ on $Y\setminus(\range(h_1)\cup\range(h_2))$.
\item On $\range(h_1)$, 
we require that $\chi(h_1(y))=\eta$
for all $y$ lying in the middle-third of an interval of type $h_1$ and
that $\chi(h_1(y))=1$ for all $y$ not lying in an interval of type $h_1$.
\item Similarly, on $\range(h_2)$, we require that $\chi(h_2(y))=\eta$
for all $y$ lying in the middle-third of an interval of type $h_2$ and
that $\chi(h_2(y))=1$ for all $y$ not lying in an interval of type $h_2$.
\end{itemize}
Since $\diam(I_j)\ge \delta/|b|$, we can choose $\chi$ to be $C^1$ with $|\chi'|\le
{\BIG\frac{3(1-\eta)|b|}{\delta P}}$ where
$P=\min_{m=1,2}\{\inf|h_m'|\}$.
From now on, we choose $\eta\in[\eta_0,1)$ sufficiently close to $1$ that
$|\chi'|\le |b|$.

\begin{cor} \label{cor-cancel}
Let $\delta$, $\Delta$ be as in Lemma~\ref{lem-cancel}.
Let $|b|>4\pi/D$, $(u,v)\in\cC_b$.
Let $\chi=\chi(b,u,v)$ be the $C^1$ function described above (using the
branches $h_1,h_2\in\cH_{n_0}$ from (UNI)).
Then $|L_s^{n_0}v|\le 
L_\sigma^{n_0}(\chi u)$ for all $s=\sigma+ib$, $|\sigma|<\eps$.
\end{cor}

Let $\hat I=\bigcup_{j=0}^{N-1}\hat I_j$ where $\hat I_j$ denotes the middle-third of $I_j$.
Each gap $J_j$, $j=1\dots,N-1$, lies between two intervals $I_{j-1}$ and $I_j$.
Let $\hat J_j$ be the interval consisting of $J_j$ together with the rightmost third of $I_{j-1}$ and the leftmost third of $I_j$.
(Also we define $\hat J_0$ and $\hat J_N$ with the obvious modifications.)
Set $\hat J=\bigcup_{j=0}^N \hat J_j$.
Then $Y=\hat I\cup \hat J$.
By construction, $\diam(\hat I_j)\ge \frac13 \delta/|b|$ and 
$\diam(\hat J_j)\le (\frac43 \delta+2\Delta)/|b|$.
In particular, there is a constant $\delta'=\delta/(4\delta+6\Delta)>0$
(independent of $b$) such that
$\diam(\hat I_j)\ge\delta' \diam(\hat J_j)$ for $j=0,\dots,N-1$.
Since $d\mu/d\Leb$ is bounded above and below, there is a constant $\delta''>0$ such that
$\mu(\hat I_j)\ge \delta''\mu(\hat J_j)$.

\begin{prop} \label{prop-fed}
Suppose that $w>0$ is a $C^\alpha$ function with $|\log w|_\alpha\le K|b|^\alpha$.
Then 
$\int_{\hat I}w\,d\mu\ge \delta'''
\int_{\hat J}w\,d\mu$,
where $\delta'''=\frac12\delta''\exp\{-(2\delta+2\Delta)^\alpha K\}$.
\end{prop}

\begin{proof}
Let $x\in \hat I_j$, $y\in \hat J_j$.
Then $|x-y|\le (2\delta+2\Delta)/|b|$
and so $|w(x)/w(y)|\le e^{K'}$ where
$K'= (2\delta+2\Delta)^\alpha K$.
It follows that
\begin{align*}
\int_{\hat I_j}w\,d\mu & \ge\mu(\hat I_j)\inf_{\hat I_j} w \ge
 \delta'' e^{-K'} \mu(\hat J_j)\sup_{\hat J_j} w 
=2 \delta'''  \mu(\hat J_j)\sup_{\hat J_j} w 
\ge2 \delta'''\int_{\hat J_j}w\,d\mu,
\end{align*}
and the result follows.
(The factor $2$ takes care of the extra interval $\hat J_N$).
\end{proof}

\subsection{Invariance of cone condition}

\begin{lemma} \label{lem-cone}
There is a constant $C_4$ depending only on $C_1$, $C_2$, $|f_0^{-1}|_\infty$
and $|f_0|_\alpha$ such that for
$n_0$ satisfying~\eqref{eq-large2} the following holds:
 
For all $(u,v)\in \cC_b$, we have that
\[
\bigl(\,L_\sigma^{n_0}(\chi u)\,,\,L_s^{n_0}v\,\bigr)\in \cC_b,
\]
for all $s=\sigma+ib$, $|\sigma|<\eps$, $|b|\ge1$.
(Here, $\chi=\chi(b,u,v)$ is from Corollary~\ref{cor-cancel}.)
\end{lemma}

\begin{proof}
Let $\hat u=L_\sigma^{n_0}(\chi u)$,
$\hat v=L_s^{n_0}v$.
Since $\chi u\ge \eta u>0$ and $L_\sigma$ is a positive operator, we
have that $\hat u>0$.   The condition $|\hat v|\le \hat u$ follows from
Corollary~\ref{cor-cancel}.
It remains to show that 
$|\log \hat u|_\alpha\le C_4|b|^\alpha$ and that
$|\hat v(x)-\hat v(y)|\le C_4|b|^\alpha \hat u(y)|x-y|^\alpha$.

Now 
\[
\frac{\hat u(x)}{\hat u(y)}=\frac{f_\sigma(y)}{f_\sigma(x)}\frac
{\sum_{h\in\cH_{n_0}}(A_{\sigma,h,n_0}(f_\sigma \chi u))(x)}
{\sum_{h\in\cH_{n_0}}(A_{\sigma,h,n_0}(f_\sigma \chi u))(y)}.
\]
Note that $|\log f_\sigma|_\alpha\le 4|f_0^{-1}|_\infty|f_0|_\alpha|x-y|^\alpha$.
Hence $f_\sigma(y)/f_\sigma(x)\le \exp\{4|f_0^{-1}|_\infty|f_0|_\alpha|x-y|^\alpha\}$.

Recall that $\chi\in[\frac12,1]$ and $|\chi'|\le|b|$.
Hence $|(\log\chi)'|\le 2|b|$ so that
$|\log \chi(x)-\log\chi(y)|\le 2|b||x-y|$.
Also, since $\chi\in[\frac12,1]$, we have
$|\log \chi(x)-\log\chi(y)|\le \log 2<1$.
Hence
\[
|\log \chi(x)-\log\chi(y)|\le 2\min\{1,|b||x-y|\}\le
2|b|^\alpha|x-y|^\alpha.
\]

We compute that
\begin{align*}
& \Bigl|\frac{(A_{\sigma,h,n_0}(\chi u))(x)}{(A_{\sigma,h,n_0}(\chi u))(y)}\Bigr|
  =\Bigl|\frac{e^{-\sigma R_{n_0}\circ h(x)}}{e^{-\sigma R_{n_0}\circ h(y)}}
 \frac{h'x}{h'y} 
\frac{f_\sigma(hx)}{f_\sigma(hy)}
\frac{\chi(hx)}{\chi(hy)} 
\frac{u(hx)}{u(hy)}\Bigr|
\\ & \qquad\qquad \le
\exp\{C_2|x-y|\} \exp\{C_2|x-y|^\alpha\} \exp\{4|f_0^{-1}|_\infty|f_0|_\alpha C_1^\alpha|x-y|^\alpha\}
\\ & \qquad\qquad\qquad\qquad\qquad \qquad \times \exp\{2C_1^\alpha|b|^\alpha|x-y|^\alpha\} \exp\{C_1^\alpha C_4\rho^{n_0}|b|^\alpha|x-y|^\alpha\}.
\end{align*}
Let $C_4=8|f_0^{-1}|_\infty|f_0|_\alpha C_1 +5C_2$
and choose $n_0$ as in~\eqref{eq-large2}.  In particular, $C_1^\alpha C_4\rho^{n_0}< 1\le C_1<C_2$.  Then
\[
\log \frac{\hat u(x)}{\hat u(y)}\le 
(8|f_0^{-1}|_\infty|f_0|_\alpha C_1 +5C_2) |b|^\alpha|x-y|^\alpha=
C_4|b|^\alpha|x-y|^\alpha,
\]
so we obtain that $|\log \hat u|_\alpha\le C_4|b|^\alpha$.

The verification that
$|\hat v(x)-\hat v(y)|\le C_4|b|^\alpha \hat u(y)|x-y|^\alpha$ involves a calculation similar to the one in the proof of Lemma~\ref{lem-LY}, though it is convenient to reorder the terms slightly.
First write
\begin{align*}
\hat v(x)-\hat v(y) & =(L_s^{n_0}v)(x)- (L_s^{n_0}v)(y) = I_1+I_2
\end{align*}
where
\begin{align*}
I_1 & =
\lambda_\sigma^{-n_0}f_\sigma(x)^{-1}\Bigl\{\Bigl(\sum_{h\in\cH_{n_0}}A_{s,h,n_0}(f_\sigma v)\Bigr)(x)-\Bigl(\sum_{h\in\cH_{n_0}}A_{s,h,n_0}(f_\sigma v)\Bigr)(y)\Bigr\},
  \\ I_2 & =(f_\sigma(y)f_\sigma(x)^{-1}-1)\lambda_\sigma^{-n_0}f_\sigma(y)^{-1}
\Bigl(\sum_{h\in\cH_{n_0}}A_{s,h,n_0}(f_\sigma v)\Bigr)(y).
\end{align*}
Now
\begin{align*}
(A_{s,h,n_0}v)(x)-(A_{s,h,n_0}v)(y) & =
(e^{-\sigma R_{n_0}\circ h(x)}-
e^{-\sigma R_{n_0}\circ h(y)})e^{-ibR_{n_0}\circ h(x)}|h'x|v(hx)
\\ & + e^{-\sigma R_{n_0}\circ h(y)}e^{-ibR_{n_0}\circ h(x)}|h'x|(v(hx)-v(hy))
\\ & + e^{-\sigma R_{n_0}\circ h(y)}e^{-ibR_{n_0}\circ h(x)}(|h'x|-|h'y|)v(hy)
\\ & +e^{-\sigma R_{n_0}\circ h(y)}(e^{-ibR_{n_0}\circ h(x)}-e^{-ibR_{n_0}\circ h(y)})|h'y|v(hy) 
\\ & = J_1+J_2+J_3+J_4.
\end{align*}
Since $(u,v)\in\cC_b$ and $|b|\ge1$, it follows from~\eqref{eq-large2} that
\begin{align*}
|v(hx)| & \le 
|v(hy)|+C_4|b|^\alpha u(hy)C_1^\alpha\rho^{n_0}|x-y|^\alpha\le
u(hy)+|b|^\alpha u(hy)
\\ & \le 2|b|^\alpha u(hy).
\end{align*}
Hence using~(iii$_1$) and~\eqref{eq-h'},
\begin{align*}
|J_1| & \le e^{-\sigma R_{n_0}\circ h(y)}C_2|\sigma||x-y||h'x||v(hx)|
\\ &  \le 2|b|^\alpha e^{-\sigma R_{n_0}\circ h(y)}C_2(2C_2+1)|x-y||h'y|u(hy)
\\ &  \le 2C_2(2C_2+1)|b|^\alpha (A_{\sigma,h,n_0}u)(y) |x-y|^\alpha.
\end{align*}
Similarly,
$|J_2|\le (2C_2+1)|b|^\alpha (A_{\sigma,h,n_0}u)(y) |x-y|^\alpha$, 
$|J_3|\le 2C_2(A_{\sigma,h,n_0}u)(y) |x-y|^\alpha$, 
$|J_4|\le 2C_2^\alpha|b|^\alpha (A_{\sigma,h,n_0}u)(y) |x-y|^\alpha$.
Altogether, 
\[
| (A_{s,h,n_0}v)(x)-(A_{s,h,n_0}v)(y)|\le 
C'|b|^\alpha (A_{\sigma,h,n_0}u)(y) |x-y|^\alpha.
\]
Since $|f_\sigma|_\infty\le 2|f_0|_\infty$,
$|f_\sigma^{-1}|_\infty\le 2|f_0^{-1}|_\infty$ 
and $\chi\ge\frac12$, it follows that
\[
| (A_{s,h,n_0}(f_\sigma v))(x)-(A_{s,h,n_0}(f_\sigma v))(y)|\le 
8C'|f_0|_\infty|f_0^{-1}|_\infty|b|^\alpha (A_{\sigma,h,n_0}(f_\sigma \chi u))(y) |x-y|^\alpha.
\]
Hence
\begin{align*}
|I_1| & \le 4|f_0|_\infty|f_0^{-1}|_\infty\lambda_\sigma^{-n_0}f_\sigma(y)^{-1}
\sum_{h\in\cH_{n_0}}
| (A_{s,h,n_0}(f_\sigma v))(x)-(A_{s,h,n_0}(f_\sigma v))(y)| \\
& \le 32C'|f_0|_\infty^2|f_0^{-1}|_\infty^2|b|^\alpha \hat u(y)|x-y|^\alpha.
\end{align*}
A simpler calculation gives
\begin{align*}
|I_2| & \le |f_\sigma(y)f_\sigma(x)^{-1}-1|(L_\sigma^{n_0}(|v|))(y)
 \le 2|f_\sigma(y)f_\sigma(x)^{-1}-1|(L_\sigma^{n_0}(\chi u))(y)
\\ &  \le 8|f_0^{-1}|_\infty|f_0|_\alpha|x-y|^\alpha \hat u(y).
\end{align*}
Hence
$|\hat v(x)-\hat v(y)| \le C_4|b|^\alpha \hat u(y)|x-y|^\alpha$
as required.
\end{proof}

\subsection{$L^2$ contraction}

\begin{lemma} \label{lem-L2}
There exist $\eps,\beta\in(0,1)$ such that
\[
\int |L_s^{mn_0}v|^2\,d\mu\le \beta^m|v|_\infty^2
\]
for all $m\ge1$, $s=\sigma+ib$, $|\sigma|<\eps$, $|b|\ge\max\{4\pi/D,1\}$, and all $v\in C^\alpha(Y)$ satisfying $|v|_\alpha\le C_4|b|^\alpha|v|_\infty$.
\end{lemma}

\begin{proof}
Define
$u_0\equiv1, v_0=v/|v|_\infty$ and inductively,
\[
u_{m+1}=L_\sigma^{n_0}(\chi_mu_m), \qquad v_{m+1}=L_s^{n_0}(v_m),
\]
where $\chi_m=\chi(b,u_m,v_m)$.
It is immediate from the definitions that $(u_0,v_0)\in\cC_b$, and it follows
from Lemma~\ref{lem-cone} that $(u_m,v_m)\in\cC_b$ for all $m$.
Hence inductively the $\chi_m$ are well-defined as in Corollary~\ref{cor-cancel}.

We will show that there exists $\beta\in(0,1)$ such that
$\int u_{m+1}^2\,d\mu\le\beta\int u_m^2\,d\mu$ for all $m$.
Then
\[
|L_s^{mn_0}v|=
|v|_\infty|L_s^{mn_0}v_0|
=|v|_\infty|v_m|
\le |v|_\infty u_m,
\]
so that
\[
\int |L_s^{mn_0}v|^2\,d\mu
\le |v|_\infty^2\int u_m^2\,d\mu
\le \beta^m|v|_\infty^2\int u_0^2\,d\mu
= \beta^m|v|_\infty^2,
\]
as required.

Now
\begin{align*}
 u_{m+1}
 & =
\lambda_\sigma^{-n_0}f_\sigma^{-1}
\sum_{h\in\cH_{n_0}}e^{-\sigma R_{n_0}\circ h}|h'|(f_\sigma \chi_m u_m)\circ h
\\ & =
\lambda_\sigma^{-n_0}f_\sigma^{-1}
\sum_{h\in\cH_{n_0}}\{|h'|^{1/2}f_\sigma^{1/2} u_m)\circ h\}
\{e^{-\sigma R_{n_0}\circ h}|h'|^{1/2}(f_\sigma^{1/2} \chi_m )\circ h\},
\end{align*}
so by Cauchy-Schwarz
\begin{align*}
 u_{m+1}^2
 & \le
(\lambda_\sigma^{n_0}f_\sigma)^{-2}
\sum_{h\in\cH_{n_0}}|h'|(f_\sigma u_m^2)\circ h
\;\sum_{h\in\cH_{n_0}}e^{-2\sigma R_{n_0}\circ h}|h'|(f_\sigma\chi_m^2 )\circ h
 \\ & \le
(\lambda_\sigma^{n_0}f_\sigma)^{-2}
\Bigl|\frac{f_\sigma}{f_0}\Bigr|_\infty
\Bigl|\frac{f_\sigma}{f_{2\sigma}}\Bigr|_\infty
\sum_{h\in\cH_{n_0}}|h'|(f_0 u_m^2)\circ h
\;\sum_{h\in\cH_{n_0}}e^{-2\sigma R_{n_0}\circ h}|h'|(f_{2\sigma}\chi_m^2 )\circ h
 \\ & \le \xi(\sigma)
L_0^{n_0}(u_m^2)\;
L_{2\sigma}^{n_0}(\chi_m^2)
\end{align*}
where (noting that $\lambda_0=1$),
\[
\xi(\sigma)=
(\lambda_\sigma^{-2} \lambda_{2\sigma})^{n_0}
\Bigl|\frac{f_0}{f_\sigma}\Bigr|_\infty
\Bigl|\frac{f_{2\sigma}}{f_\sigma}\Bigr|_\infty
\Bigl|\frac{f_\sigma}{f_0}\Bigr|_\infty
\Bigl|\frac{f_\sigma}{f_{2\sigma}}\Bigr|_\infty.
\]

As in Subsection~\ref{sec-cancel}, we write $Y=\hat I\cup\hat J$.
If $y\in\hat I$ then $y$ lies in the middle third of an interval of type $h_1$ or $h_2$.
Suppose without loss that the type is $h_1$.  Then $\chi_m(h_1(y))=\eta$.
Hence
\begin{align*}
& (L_{2\sigma}^{n_0} \chi_m^2)(y)  \le 
\lambda_{2\sigma}^{-n_0}f_{2\sigma}(y)^{-1}
\Big\{\eta^2 e^{-2\sigma R_{n_0}\circ h_1(y)}|h_1'(y)|f_{2\sigma}(h_1(y))
\\ & \qquad \qquad \qquad
 \qquad \qquad \qquad\qquad\qquad +
\sum_{h\in\cH_{n_0},\,h\neq h_1}
e^{-2\sigma R_{n_0}\circ h(y)}|h'y|f_{2\sigma}(hy)\Bigr\}
\\ & \qquad = 
 (L_{2\sigma}^{n_0} 1)(y) -(1-\eta^2) \lambda_{2\sigma}^{-n_0}f_{2\sigma}(y)^{-1} e^{-2\sigma R_{n_0}\circ h_1(y)}|h_1'(y)|f_{2\sigma}(h_1(y)) 
\\ & \qquad \le  
1-(1-\eta^2)2^{-(n_0+2)}
|f_0|_\infty^{-1}\inf f_0\,
 e^{-2|R_{n_0}\circ h_1|_\infty}\inf |h_1'|= \eta_1<1.
\end{align*}
In this way we obtain that there exists $\eta_1<1$ such that
\[
u_{m+1}^2(y)\le \begin{cases} \xi(\sigma)\eta_1(L_0^{n_0} u_m^2)(y), &
y\in\hat I \\
\xi(\sigma)(L_0^{n_0} u_m^2)(y), &  y\in\hat J \end{cases}.
\]

Since $(u_m,v_m)\in\cC_b$ it follows in particular that 
$|\log u_m|_\alpha\le C_4|b|^\alpha$.
Hence $u_m^2(hx)/u_m^2(hy)\le 
\exp\{2C_4C_1^\alpha\rho^{n_0}|b|^\alpha|x-y|^\alpha\}
\le \exp\{|b|^\alpha|x-y|^\alpha\}$ 
by~\eqref{eq-large2}.
Let $w=L_0^{n_0}(u_m^2)$.
Then
\[
\frac{w(x)}{w(y)} = 
\frac{f_0(y)\sum_{h\in\cH_{n_0}}|h'x|f_0(hx)u_m^2(hx)}
{f_0(x)\sum_{h\in\cH_{n_0}}|h'y|f_0(hy)u_m^2(hy)}
\]
where
\[
\frac{|h'x|f_0(hx)u_m^2(hx)}{|h'y|f_0(hy)u_m^2(hy)}
\le \exp\{C_2|x-y|^\alpha\}\exp\{|f_0^{-1}|_\infty|f_0|_\alpha|x-y|^\alpha\}\exp\{|b|^\alpha|x-y|^\alpha\}.
\]
Hence $|\log w|_\alpha \le K|b|^\alpha$ where $K=2|f_0^{-1}|_\infty|f_0|_\alpha+2C_2$,
and so $w$ satisfies the hypotheses of Proposition~\ref{prop-fed}.
Consequently, $\int_{\hat I}w\,d\mu\ge \delta''' \int_{\hat J}w\,d\mu$.
Let $\beta'={\BIG\frac{1+\eta_1\delta'''}{1+\delta'''}}<1$.
Then
$\delta'''={\BIG\frac{1-\beta'}{\beta'-\eta_1}}$ and so
\[
(\beta'-\eta_1)\int_{\hat I}w\,d\mu\ge (1-\beta')
\int_{\hat J}w\,d\mu,
\]
which rearranges to give
\[
\eta_1\int_{\hat I}w\,d\mu
+ \int_{\hat J}w\,d\mu\le\beta'\int_Y w\,d\mu.
\]

Hence
\begin{align*}
\int_Y u_{m+1}^2\,d\mu  & \le 
\xi(\sigma) \Bigl(\eta_1\int_{\hat I}L_0^{n_0} (u_m^2)\,d\mu
+ \int_{\hat J}L_0^{n_0} (u_m^2)\,d\mu\Bigr) \\ & \le
\xi(\sigma)\beta'\int_Y L_0^{n_0} (u_m^2)\,d\mu
=\xi(\sigma)\beta'\int_Y u_m^2\,d\mu.
\end{align*}
Finally by Remark~\ref{rmk-P} we can shrink $\eps$ if necessary so that $\xi(\sigma)\beta'\le\beta<1$ for $|\sigma|<\eps$.~
\end{proof}

\subsection{From $L^2$ contraction to $C^\alpha$ contraction}

The next result shows how to pass from $L^1$ estimates to $L^\infty$ estimates.

\begin{prop} \label{prop-L1}
For any $B>1$, there exists $\eps\in(0,1)$, $\tau\in(0,1)$, $C_5>0$ such that
\begin{align*}
|L_s^nv|_\infty^2 \le 
C_5(1+|b|^\alpha) \tau^n|v|_\infty\|v\|_b
+C_5 B^n|v|_\infty \int|v|\,d\mu,
\end{align*}
for all $s=\sigma+ib$, $|\sigma|<\eps$, $n\ge1$, $v\in C^\alpha(Y)$.
\end{prop}

\begin{proof}
Modifying the Cauchy-Schwarz argument in the proof of Lemma~\ref{lem-L2},
\begin{align*}
|L_s^nv| & \le  
\lambda_\sigma^{-n}f_\sigma^{-1}
\sum_{h\in\cH_n}e^{-\sigma R_n\circ h}|h'|(f_\sigma |v|)\circ h
\\ & =
\lambda_\sigma^{-n}f_\sigma^{-1}
\sum_{h\in\cH_n}
\{e^{-\sigma R_n\circ h}|h'|^{1/2}(f_\sigma |v|)^{1/2}\circ h \}
\{|h'|^{1/2}(f_\sigma|v|)^{1/2}\circ h\}
\end{align*}
so 
\begin{align*}
|L_s^nv|^2
 & \le
\lambda_\sigma^{-2n}f_\sigma^{-2}
\sum_{h\in\cH_n}e^{-2\sigma R_n\circ h}|h'|(f_\sigma|v|)\circ h
\;\sum_{h\in\cH_n}|h'|(f_\sigma|v|)\circ h
 \\ & \le
(\lambda_\sigma^{-2}\lambda_{2\sigma})^n\xi(\sigma) L_{2\sigma}^n (|v|) L_0^n (|v|)
\end{align*}
where $\xi(\sigma)=
|f_0/f_\sigma|_\infty
|f_{2\sigma}/f_\sigma|_\infty
|f_\sigma/f_0|_\infty
|f_\sigma/f_{2\sigma}|_\infty\le 16$.
By Remark~\ref{rmk-P}, 
\begin{align} \label{eq-C5}
|L_s^nv|_\infty^2 \le 16
B^n|v|_\infty |L_0^n (|v|)|_\infty,
\end{align}
where $B$ is arbitrarily close to $1$.
Since $L_0$ is the normalised transfer operator for the uniformly expanding map
$F:Y\to Y$, there are constants $C'>0$, $\tau_1\in(0,1)$ such that
$|L_0^nw|_\infty\le C'\tau_1^n \|w\|_\alpha$ for all $w\in C^\alpha(Y)$ with $\int w\,d\mu=0$
and all $n\ge1$.   
Note that $\|w\|_{\alpha}\le (1+|b|^\alpha)\|w\|_b$ for all $b$.

Taking $w=|v|-\int|v|\,d\mu$, we obtain that
$|L_0^n(|v|)-\int|v|\,d\mu|_\infty\le 2C'\tau_1^n \|v\|_{\alpha}$ and hence
$|L_0^n(|v|)|_\infty \le 2C'(1+|b|^\alpha)\tau_1^n\|v\|_b+\int|v|\,d\mu$.
Substituting into~\eqref{eq-C5}, we obtain 
\begin{align*}
|L_s^nv|_\infty^2 \le 
32C'(1+|b|^\alpha) (B\tau_1)^n|v|_\infty\|v\|_b
+16 B^n|v|_\infty \int|v|\,d\mu.
\end{align*}
Finally, shrink $B>1$ if necessary so that $\tau=B\tau_1<1$.
\end{proof}

\begin{cor} \label{cor-L2}
There exists $\eps\in(0,1)$, $A>0$ and
$\beta\in(0,1)$ such that
\[
\|L_s^{4mn_0}v\|_b\le \beta^m\|v\|_b
\]
for all $m\ge A\log|b|$, $s=\sigma+ib$, $|\sigma|<\eps$, $|b|\ge\max\{4\pi/D,1\}$, and all $v\in C^\alpha(Y)$ satisfying $|v|_\alpha\le C_4|b|^\alpha|v|_\infty$.
\end{cor}

\begin{proof}
Substituting $L_s^nv$ in place of $v$ in Proposition~\ref{prop-L1}
and applying Corollary~\ref{cor-LY},
\begin{align*}
|L_s^{2n}v|_\infty^2  & \le 
2C_3C_5(1+|b|^\alpha) \tau^n|v|_\infty\|v\|_b
+C_5 B^n|v|_\infty \Bigl(\int|L_s^nv|^2\,d\mu\Bigr)^{1/2}.
\end{align*}
By Lemma~\ref{lem-L2},
\begin{align*}
|L_s^{2mn_0}v|_\infty^2  & \le 
2C_3C_5(1+|b|^\alpha) \tau^{mn_0}|v|_\infty\|v\|_b
+C_5 B^{mn_0}|v|_\infty \beta^{m/2}|v|_\infty,
\end{align*}
for all $m\ge1$.
Shrinking $B>1$ if necessary 
there exists $\beta_1<1$ such that
\begin{align*}
|L_s^{2mn_0}v|_\infty^2  & \le 
4C_3C_5(1+|b|^\alpha) \beta_1^m|v|_\infty\|v\|_b.
\end{align*}
Hence, there exists $A>0$, $\beta_2<1$, such that
$|L_s^{2mn_0}v|_\infty^2  \le 
\beta_2^{2m}|v|_\infty\|v\|_b$ for all $m\ge A\log|b|$, and so
\begin{align} \label{eq-infty}
|L_s^{2mn_0}v|_\infty  \le 
\beta_2^m\|v\|_b \quad\text{ for all $m\ge A\log|b|$.}
\end{align}

Next, substituting $L_s^nv$ for $v$ in Lemma~\ref{lem-LY},
\begin{align*}
|L_s^{2n}v|_\alpha & \le C_3(1+|b|^\alpha)\{
 |L_s^nv|_\infty+\rho^n
 \|L_s^nv\|_b\}
\le 2C_3^2(1+|b|^\alpha)\{ |L_s^nv|_\infty+ \rho^n \|v\|_b\}.
\end{align*}
Taking $n=2mn_0$, and using~\eqref{eq-infty},
\begin{align*}
|L_s^{4mn_0}v|_\alpha & \le
2C_3^2(1+|b|^\alpha)\{\beta_2^m \|v\|_b+\rho^{2mn_0} \|v\|_b\}
\le 4C_3^2(1+|b|^\alpha)\beta_3^m \|v\|_b
\end{align*}
for all $m\ge A\log|b|$.
This combined with~\eqref{eq-infty} shows that
$\|L_s^{4mn_0}v\|_b\le 4C_3^2\beta_3^m\|v\|_b$ for all
$m\ge A\log|b|$.  Finally the choices of $\beta_3$ and $A$ can be modified to
absorb the constant $4C_3^2$.
\end{proof}

\begin{thm} \label{thm-dolg}
Let $D'=\max\{4\pi/D,2\}$.
There exists $\eps\in(0,1)$, $\gamma\in(0,1)$ and $A>0$
such that $\|P_s^n\|_b\le \gamma^n$ for
all $s=\sigma+ib$, $|\sigma|<\eps$, $|b|\ge D'$, $n\ge A\log |b|$.
\end{thm}

\begin{proof}
We claim that 
there exists $\eps\in(0,1)$, $\gamma_1\in(0,1)$ $A,C>0$
such that
$\|L_s^{4mn_0}\|_b\le C\gamma_1^m$ for
all $s=\sigma+ib$, $|\sigma|<\eps$, $|b|\ge \max\{4\pi/D,2\}$, $m\ge A\log |b|$.

Suppose that the claim holds.  Write $n=4mn_0+r$ where $r<4n_0$.
By Corollary~\ref{cor-LY}, 
\[
\|L_s^n\|_b\le \|L_s^r\|_b\|L_s^{4mn_0}\|_b
\le 2C_3C\gamma_1^m \ll (\gamma_1^{1/(4n_0)})^n.
\]
By definition, $P_sv=\lambda_\sigma f_\sigma L_s(f_\sigma^{-1}v)$
so using the fact that $\|f_\sigma\|_\alpha$ and 
$\|f_\sigma^{-1}\|_\alpha$ are bounded for $|\sigma|<\eps$, we obtain from~\eqref{eq-alg} that
$\|P_s^n\|_b\ll \lambda_\sigma^n \|L_s^n\|_b
\ll ({\gamma_1^{1/(4n_0)}\lambda_\sigma})^n$.
By Remark~\ref{rmk-P} we can arrange that 
${\gamma_1^{1/(4n_0)}\lambda_\sigma}\le \gamma<1$ for
$|\sigma|<\eps$.  Then 
$\|P_s^n\|_b\le C\gamma^n$ for all $n\ge A\log|b|$.
Finally, we can increase $A$ and modify $\gamma$ to absorb the constant $C$
proving the theorem.

The verification of the claim splits into two cases.
In the harder case 
$|v|_\alpha \le C_4|b|^\alpha|v|_\infty$,
the claim follows from Corollary~\ref{cor-L2}.

It remains to deal with the simpler case where 
$|v|_\alpha > C_4|b|^\alpha|v|_\infty$.
Recall that $C_4\ge 6C_3\ge 6$ so 
\[
|L_s^{n_0}v|_\infty\le |v|_\infty
<(C_4|b|^\alpha)^{-1}|v|_\alpha
\le (C_4|b|^\alpha)^{-1}(1+|b|^\alpha)\|v\|_b
\le 2C_4^{-1}\|v\|_b
\le {\SMALL\frac13}\|v\|_b.
\]
By Lemma~\ref{lem-LY} and~\eqref{eq-large3}, 
\begin{align*}
|L_s^{n_0}v|_\alpha & \le (1+|b|^\alpha)\{C_3|v|_\infty+C_3\rho^{n_0}\|v\|_b\}
\\ & \le (1+|b|^\alpha)\{2C_3C_4^{-1}\|v\|_b+{\SMALL\frac13}\|v\|_b\}
= {\SMALL\frac23}(1+|b|^\alpha)\|v\|_b.
\end{align*}
Hence $\|L_s^{n_0}\|_b\le \frac23$.
\end{proof}

\subsection{Proof of Theorem~\ref{thm-semiflow}}

In this subsection, we show to proceed from Theorem~\ref{thm-dolg} to the main result.   

Define the Laplace transform $\hat\rho_{v,w}(s)=\int_0^\infty e^{-st}\rho_{v,w}(t)\,dt$.  
The key estimate is the following:

\begin{lemma} \label{lem-semiflow}
There exists $\eps>0$ such that $\hat\rho_{v,w}$ is analytic on $\{\Re s>-\eps\}$ for all $v\in F_\alpha(Y^R)$, $w\in L^\infty(Y^R)$.
Moreover, there is a constant $C>0$ such that
$|\hat\rho_{v,w}(s)|\le C(1+|b|^{1/2})\|v\|_\alpha|w|_\infty$ for
all $s=\sigma+ib$ with $\sigma\in[-\frac12\eps ,0]$.
\end{lemma}

\begin{pfof}{Theorem~\ref{thm-semiflow}}
By Lemma~\ref{lem-semiflow}, $\hat\rho_{v,w}$ is
analytic on $\{\Re s>-\eps\}$.
The inversion formula gives
\[
\rho_{v,w}(t) = \int_\Gamma e^{st} \hat\rho_{v,w}(s)\,ds,
\]
where we can take $\Gamma=\{\Re s=-\frac12\eps\}$.

Applying Taylor's Theorem as in~\cite[Section 4, VI]{Dolgopyat98a},
\[
\hat\rho_{v,w}(s)=\rho_{v,w}(0)s^{-1}+
\rho_{\partial_tv,w}(0)s^{-2}+s^{-2}\hat\rho_{\partial_t^2v,w}(s).
\]
By Lemma~\ref{lem-semiflow},
$|s^{-2}\hat\rho_{\partial_t^2v,w}(s)|\ll 
|s|^{-2}(1+|b|^{1/2})\|v\|_{\alpha,2}|w|_\infty
\ll (1+|b|^{3/2})^{-1}\|v\|_{\alpha,2}|w|_\infty$ 
for
$\sigma=-\frac12\eps$ and the result follows.
\end{pfof}

In the remainder of this section, we prove Lemma~\ref{lem-semiflow}.
Given $v,w\in L^\infty(Y^R)$, $s\in\C$, define 
\[
v_s(y)=\int_0^{R(y)}e^{su}v(y,u)\,du, \qquad
w_s(y)=\int_0^{R(y)}e^{-su}w(y,u)\,du.
\]
Let $\hat J_n(s)=\int_Y e^{-sR_n}v_s\,w_s\circ F^n\,d\mu$.
By Appendix~\ref{sec-Poll},
$\hat\rho_{v,w}(s)  =J_0(s)+(1/\bar R)\Psi(s)$
where $|J_0(s)|\ll |v|_\infty|w|_\infty$
and $\Psi(s)=\sum_{n=1}^\infty \hat J_n(s)$.

Let $A$, $D'$ be as in Theorem~\ref{thm-dolg}.
We split the proof into three ranges of $n$ and $b$: (i) $|b|\le D'$, (ii)
$n\le A\log |b|$, $|b|\ge2$, and (iii) $|b|\ge D'$, $n\ge A\log |b|$.
Lemma~\ref{lem-semiflow} follows from
Lemmas~\ref{lem-semiflowi},~\ref{lem-semiflowii},~\ref{lem-semiflowiii} below.

\begin{prop} \label{prop-ReR}
$Re^{\frac12 \eps R}\le 2\eps^{-1}e^{\eps R}$
and $\int_Y e^{\eps R}\,d\mu<\infty$.
\end{prop}

\begin{proof}  The first statement follows from the inequality
$te^t\le e^{2t}$ which holds for all $t\in\R$.
By~\eqref{eq-diam},
\[
\int_Y e^{\eps R}\,d\mu
=\sum_{h\in\cH}\int_{h(Y)}e^{\eps R}\,d\mu
\le \sum_{h\in\cH}\diam h(Y)e^{\eps |R\circ h|_\infty}
\le e^{C_2}
\sum_{h\in\cH} e^{\eps|R\circ h|_\infty}|h'|_\infty
\]
 which is finite by condition~(iv).
\end{proof}

\begin{lemma}[The range $n\le A\log|b|$, $|b|\ge2$.] \label{lem-semiflowi}
There exists $\eps>0$, $C>0$ such that
\[
\sum_{1\le n\le A\log|b|}|\hat J_n(s)|\le C\eps^{-2}(1+|b|^{1/2})|v|_\infty|w|_\infty\]
for all
$v,w\in L^\infty(Y^R)$ and for all
all $s=\sigma+ib$ with $\sigma\in[-\frac12\eps ,0]$, $|b|\ge 2$.
\end{lemma}

\begin{proof}  Note that $|v_s(y)|\le R(y)|v|_\infty$ and
$|w_s(y)|\le R(y) e^{\frac12\eps R(y)}|w|_\infty
\le 2\eps^{-1} e^{\eps R(y)}|w|_\infty$.
Hence
\[
|\hat J_n(s)|
\le 2\eps^{-1}|v|_\infty|w|_\infty\int_Y e^{\frac12\eps R_n}R\, e^{\eps R}\circ F^n\,d\mu.
\]
It is convenient to introduce the 
normalised twisted transfer operators $Q_s$ given by $\int_Y Q_s f\,g\,d\mu=\int_Y e^{-sR}f\,g\circ F\,d\mu$ for $f\in L^\infty(Y)$, $g\in L^1(Y)$.
Note that $Q_sf=f_0^{-1}P_s(f_0f)$ and hence (like $P_s$) has spectral radius
at most $\lambda_\sigma$.
Hence
\begin{align*}
 \int_Y e^{\frac12\eps R_n}R\, e^{\eps R}\circ F^n\,d\mu
 & \le 2\eps^{-1}\int_Y e^{\eps R_n} e^{\eps R}\circ F^n\,d\mu
 = 2\eps^{-1}\int_Y Q_\eps^n1\, e^{\eps R}\,d\mu
\\ &  \le  2\eps^{-1}|Q_\eps^n1|_\infty\int_Y e^{\eps R}\,d\mu
 \le  2\eps^{-1}\lambda_{-\eps}^n\int_Y e^{\eps R}\,d\mu.
\end{align*}
It follows from Proposition~\ref{prop-ReR} that there is a constant $C'>0$ such that
$|\hat J_n(s)|\le C'\eps^{-2}\lambda_{-\eps}^n|v|_\infty|w|_\infty$
.

Recall that $\lambda_{-\eps}=(1+\delta)$ where $\delta=\delta(\eps)\to0$ as $\eps\to0$.  We can arrange that $\delta A<\frac13$.  
Then $\lambda_{-\eps}^{A\log|b|}= |b|^{A\log(1+\delta)}\le |b|^{1/3}$.  Hence
\[
\sum_{1\le n\le A\log |b|}|\hat J_n(s)|\le AC'\eps^{-2}\log|b||b|^{1/3}|v|_\infty|w|_\infty,
\]
and the result follows.
\end{proof}

We recall the definition of the seminorm on
$F_\alpha(Y^R)$ given by $\|v\|_\alpha=|v|_\infty+|v|_\alpha$ where
$|v|_\alpha=\sup_{x\neq y}\sup_{0\le u \le R(x)\wedge R(y)}|v(x,u)-v(y,u)|/|x-y|^\alpha$.

\begin{prop} \label{prop-vs}
Let $h\in\cH$.  Then
\[
|v_s\circ h|_\infty\le |R\circ h|_\infty|v|_\infty,\qquad
\|v_s\circ h\|_\alpha\le (C_1+|R\circ h|_\infty)\|v\|_\alpha,
\]
for all $v\in F_\alpha(Y^R)$, 
and all $s=\sigma+ib$ with $\sigma\in[-\frac12\eps,0]$.
\end{prop}

\begin{proof}
The first estimate is immediate.
Also,
\[ 
|v_s(x)-v_s(y)|\le |v|_\infty |R(x)-R(y)|+|v|_\alpha \max\{R(x),R(y)\}|x-y|^\alpha,
\]
so the second estimate follows from conditions~(i) and~(iii).
\end{proof}

\begin{prop} \label{prop-vw}
There exists a constant $C_6>0$  such that
\[
|w_s|_1\le C_6 \eps^{-1}|w|_\infty, \quad
|P_s(f_0v_s)|_\infty \le C_6\eps^{-1}|v|_\infty, \quad
\|P_s(f_0v_s)\|_b \le C_6\eps^{-1}\|v\|_\alpha,
\]
for all $v\in F_\alpha(Y^R)$, 
$w\in L^\infty(Y^R)$, 
and all $s=\sigma+ib$ with $\sigma\in[-\frac12\eps,0]$.
\end{prop}

\begin{proof}
First, by Proposition~\ref{prop-ReR}.
\[
|w_s|_1=\int_Y|w_s|\,d\mu
\le\int_Y\int_0^{R(y)}e^{\frac12\eps R(y)}|w(y,u)|\,du\,d\mu
\le |w|_\infty\int_Y R e^{\frac12\eps R}\,d\mu \ll \eps^{-1}|w|_\infty.
\]

Recall that $P_sv=\sum_{h\in\cH}A_{s,h}v$ where $A_{s,h}v=e^{-sR\circ h}|h'|v\circ h$.
Hence using Proposition~\ref{prop-vs},
\begin{align*}
|A_{s,h}(f_0v_s)|_\infty
& \le e^{\frac12\eps |R\circ h|_\infty}|h'||f_0\circ h|_\infty|v_s\circ h|_\infty
\le |f_0|_\infty |v|_\infty e^{\frac12\eps |R\circ h|_\infty} |R\circ h|_\infty |h'|
\\ & \le 2\eps^{-1}|f_0|_\infty |v|_\infty e^{\eps |R\circ h|_\infty} |h'|.
\end{align*}
By condition~(iv).
\[
|P_s(f_0v_s)|_\infty \le 2\eps^{-1}|f_0|_\infty |v|_\infty\sum_{h\in\cH}e^{\eps |R\circ h|_\infty} |h'|\ll \eps^{-1}|v|_\infty.
\]

Finally, it follows from the proof of Proposition~\ref{prop-P} that
\begin{align*}
|A_{s,h}(f_0v_s)|_\alpha &
\ll (1+|b|^\alpha)e^{\frac12\eps |R\circ h|_\infty}|h'|_\infty\|v_s\circ h\|_\alpha.
\end{align*}
By Proposition~\ref{prop-vs},
\[
 |A_{s,h}(f_0v_s))|_\alpha
\ll (1+|b|^\alpha)e^{\frac12\eps |R\circ h|_\infty}|h'|_\infty|R\circ h|_\infty\|v\|_\alpha 
\ll \eps^{-1}(1+|b|^\alpha)e^{\eps |R\circ h|_\infty}|h'|_\infty\|v\|_\alpha.
\]
Hence by condition~(iv), $|P_s(f_0v_s)|_\alpha\ll \eps^{-1}(1+|b|^\alpha)\|v\|_\alpha$
and it follows that
$\|P_s(f_0v_s)\|_b\ll \eps^{-1}\|v\|_\alpha$.
\end{proof}

\begin{lemma}[The range $|b|\le D'$.] \label{lem-semiflowii}
There exists $\eps>0$, $C>0$ such that
$|\Psi(s)|\le C\eps^{-2}\|v\|_\alpha|w|_\infty$ for all
$v\in F_\alpha(Y^R)$, $w\in L^\infty(Y^R)$ and for all
$s=\sigma+ib$ with $\sigma\in[-\frac12\eps ,0]$, $|b|\le D'$.
\end{lemma}

\begin{proof}
Replacing $v$ by $v-\int_{Y^R}v\,d\mu^R$, we can suppose without loss
that $v$ lies in the space $\mathcal{B}
=\{v\in F_\alpha(Y^R):\int_Y \int_0^{R(y)}v(y,u)\,du\,d\mu=0\}$.

It is again convenient to introduce the 
normalised twisted transfer operators $Q_s:C^\alpha(Y)\to C^\alpha(Y)$ mentioned in the proof of
Lemma~\ref{lem-semiflowi}.
We have 
\[
\Psi(s)=\sum_{n=1}^\infty \int_Y Q_s^n v_s\,w_s\,d\mu
=\int_Y (I-Q_s)^{-1}(Q_sv_s)\,w_s\,d\mu
=\int_Y Z_sv\,w_s\,d\mu,
\]
where $Z_sv=(I-Q_s)^{-1}(Q_sv_s)$.

Consider the family of operators $Z_s:\mathcal{B}\to C^\alpha(Y)$.
We claim that this family is analytic on $\{\Re s>0\}$ and admits an analytic extension beyond the imaginary axis.  Shrinking $\epsilon$ if necessary, we can
ensure that $Z_s$ is analytic on the region 
$\{s\in[-\eps,0]\times[-D',D']\}$
and hence there is a constant $C'>0$ such that
$\|Z_sv\|_\alpha\le C'\|v\|_\alpha$.
It then follows from Proposition~\ref{prop-vw} that
$|\Psi(s)|\le C'\|v_s\|_\alpha|w_s|_1\le C'C_6^2\eps^{-2}\|v\|_\alpha|w|_\infty$ for all
$s\in[-\eps/2,0]\times[-D',D']$.

Note that $Q_sf=f_0^{-1}P_s(f_0f)$.
Writing $s=\sigma+ib$, the spectral radius of 
$P_s$ and hence $Q_s$ is at most $\lambda_\sigma$ where $\lambda_0=1$ and $\lambda_\sigma<1$ for $\sigma>0$.  In particular,
$Z_s$ is analytic on $\{\Re s>0\}$.
Hence to prove the claim,
it remains to show that $Z_s$ is analytic on a neighborhood of $s=ib$ for each $b\in\R$.

For $|\Re s|\le \epsilon$ (with $\eps>0$ sufficiently small), it follows from Lemma~\ref{lem-LY}
that the essential spectral radius of $L_s$, and hence $Q_s$,
is strictly less than $1$, so the spectrum close to $1$ consists only of isolated eigenvalues.  

For $s=ib$ with $b\neq0$, we have the aperiodicity property that
$1\not\in\spec Q_{ib}$.  To see this, suppose 
that $Q_{ib}f=f$ for some $f\in C^\alpha(Y)$.
By definition, $Q_s$ is the $L^2$ adjoint of $f\mapsto e^{sR}f\circ F$ and hence
$e^{ibR}f\circ F=f$.
Choose $q\ge1$ so that $|qb|>D'$ and set $\tilde b=qb$, $\tilde f=f^q$.  Then
$e^{i\tilde b R}\tilde f\circ F= \tilde f$
and hence $Q_{i\tilde b}\tilde f= \tilde f$
and $P_{i\tilde b}(f_0\tilde f)=f_0\tilde f$.
By Theorem~\ref{thm-dolg}, $f\equiv0$.

It follows that for each $b\neq0$ there is an open set $U_b$ containing
$ib$ such that $1\not\in\spec Q_s$ for all $s\in U_b$.
Hence $(I-Q_s)^{-1}:C^\alpha(Y)\to C^\alpha(Y)$ is analytic on $U_b$.
By Proposition~\ref{prop-vw}, $Z_s:\mathcal{B}\to F_\alpha(Y)$ is analytic on $U_b$.

Finally we consider the point $s=0$.  For $s$ near to $0$,
let $\pi_s$ denote the spectral projection corresponding to the eigenvalue
$\lambda_s$ for $Q_s$. In particular, $\pi_0f=\int_Y f\,d\mu$. Then
$Q_s=\lambda_s\pi_s+E_s$ where $\pi_sE_s=E_s\pi_s$ and $E_s$ is a strict contraction uniformly in $s$ near $0$.
Hence $Z_sv=\sum_{n=1}^\infty Q_s^nv_s=(1-\lambda_s)^{-1}\lambda_s\pi_s v_s+Y_sv$ where $Y_s$ is analytic
in a neighborhood of $0$.
Moreover, $\lambda_s=1+cs+O(s^2)$ where $c\neq0$, so
$Z_s$ has at worst a simple pole at $0$.   But $\pi_0v_0=\int_Y\int_0^{R(y)}v(y,u)\,du\,d\mu=0$, 
so $Z_s:\mathcal{B}\to F_\alpha(Y)$ is analytic on a neighborhood of $0$ completing the proof.
\end{proof}

\begin{lemma}[The range $|b|\le D'$, $n\ge A\log|b|$.] \label{lem-semiflowiii}
There exists $\eps>0$, $C>0$ such that
\[
\sum_{n\ge A\log|b|}|\hat J_n(s)|\le C\eps^{-2}\|v\|_\alpha|w|_\infty,
\]
for all $v\in F_\alpha(Y^R)$, $w\in L^\infty(Y^R)$ and for all
$s=\sigma+ib$ with $\sigma\in[-\frac12\eps ,0]$, $|b|\le D'$.
\end{lemma}

\begin{proof}
Using the fact that $L^n(e^{-sR_n} v)=f_0^{-1}P^n(e^{-sR_n}f_0v) =f_0^{-1}P_s^n(f_0v)$ we have
\[
\hat J_n(s) =\int_Y f_0^{-1}P_s^n(P_s(f_0v_s))\,w_s\,d\mu.
\]
By Theorem~\ref{thm-dolg},
\begin{align*} 
 \sum_{n\ge A\log|b|} |P_s^n f|_\infty   & \le
\sum_{n\ge A\log|b|} \|P_s^n\|_b \|f\|_b
\le  \sum_{n\ge A\log|b|} \gamma^n \|f\|_b \le (1-\gamma)^{-1}\ \|f\|_b,
\end{align*}
so the result follows from Proposition~\ref{prop-vw}.
\end{proof}

\section{Flows over $C^{1+\alpha}$ hyperbolic skew products}
\label{sec-flow}

In this section, we prove a result on exponential decay of correlations for a class of skew product flows satisfying UNI, by reducing to the situation in Section~\ref{sec-semiflow}.
Our treatment is analogous to~\cite{AGY06}.

\paragraph{Uniformly hyperbolic skew products}
Let $X=Y\times Z$ where $Y=[0,1]$ and $Z$ is a compact Riemannian manifold.
Let $f(y,z)=(Fy,G(y,z))$ where $F:Y\to Y$, $G:Y\times Z\to Z$ are $C^{1+\alpha}$.

We say that $f:X\to X$ is a {\em $C^{1+\alpha}$ uniformly hyperbolic skew product}
if 
	$F:Y\to Y$ is a uniformly expanding map satisfying conditions~(i) and~(ii) as in Section~\ref{sec-semiflow}, with absolutely continuous invariant probability measure $\mu$, and moreover
\begin{itemize}
	\item[(v)]  There exist constants $C>0$, $\gamma_0\in(0,1)$ such that
$|f^n(y,z)-f^n(y,z')|\le C\gamma_0^n |z-z'|$ for all $y\in Y$, $z,z'\in Z$.
\end{itemize}

Let $\pi:X\to Y$ be the projection $\pi(y,z)=y$.  This defines a semiconjugacy between $f$ and $F$.  Note that $\{\pi^{-1}(y):y\in Y\}$ is an exponentially contracting stable foliation.

\begin{prop}  \label{prop-mux}
	Given $v:X\to\R$ continuous, define $v_+,v_-:Y\to\R$ by setting
$v_+(y)=\sup_zv(y,z)$,
$v_-(y)=\inf_zv(y,z)$.
Then the limits
\[
\lim_{n\to\infty}\int_Y (v\circ f^n)_+\,d\mu\quad\text{and}\quad
\lim_{n\to\infty}\int_Y (v\circ f^n)_-\,d\mu
\]
exist and coincide for all $v$ continuous.  Denote the common limit by $\int_Xv\,d\mu_X$;
this defines an $f$-invariant ergodic probability measure $\mu_X$ on $X$.
Moreover,
$\pi_*\mu_X=\mu$.
\end{prop}

\begin{proof}
See~\cite[Section~6]{APPV09}.
\end{proof}

Recall that $L=L_0$ denotes the normalised transfer operator for $F:Y\to Y$.

\begin{prop} \label{prop-disint}
	(a) Suppose $v\in C^0(X)$.
 Then the limit
 \[
  \eta_y(v)=\lim_{n\to\infty} (L^nv_n)(y),\quad v_n(y)=v\circ f^n(y,0),
       \]
exists for almost every $y\in Y$ and defines a probability measure supported on $\pi^{-1}(y)$.
Moreover $y\mapsto \eta_y(v)$ is integrable and
\begin{align}
\label{eq:defeta}
\int_X v\,d\mu_X = \int_Y \int_{\pi^{-1}(y)}v\,d\eta_y \ d\mu(y).
\end{align}

\vspace{1ex}
\noindent(b) 
	There exists $C>0$ such that,
	for any $v\in C^\alpha(X)$, the function $y\mapsto \bar v(y) =  \int_{\pi^{-1}(y)}v\,d\eta_y$ lies in $C^\alpha(Y)$ and
	$\|\bar{v}\|_\alpha \leq C \|v\|_\alpha$.
\end{prop}

\begin{proof}  This follows from 
Propositions~3 and~6 respectively of~\cite{ButterleyMapp}.
\end{proof}

\paragraph{Hyperbolic skew product flows}
Suppose that $R:Y\to\R^+$ is $C^1$ on partition elements $(c_m,d_m)$ with
$\inf R>0$.  Define $R:X\to\R^+$ by setting
$R(y,z)=R(y)$.
Define the suspension $X^R=\{(x,u)\in X\times\R:0\le u\le R(x)\}/\sim$
where $(x,R(x))\sim(fx,0)$.  The suspension flow
$f_t:X^R\to X^R$ is given by $f_t(x,u)=(x,u+t)$ computed modulo identifications,
 with ergodic invariant probability measure 
$\mu_X^R=(\mu_X\times\Leb)/\bar R$.

We say that $f_t$ is a {\em $C^{1+\alpha}$ hyperbolic skew product flow} provided
$f:X\to X$ is a $C^{1+\alpha}$ uniformly hyperbolic skew product as above, and
$R:Y\to\R^+$ satisfies conditions~(iii) and~(iv) as in Section~\ref{sec-semiflow}.
If $F:Y\to Y$ and $R:Y\to\R^+$ satisfy condition UNI from Section~\ref{sec-semiflow}, then we say that the skew product flow $f_t$ satisfies UNI.

\paragraph{Function space}
Define $F_\alpha(X^R)$ to consist of $L^\infty$ functions
$v:X^R\to \R$ such that 
$\|v\|_\alpha=|v|_\infty+|v|_\alpha<\infty$ where
\[
|v|_\alpha=\sup_{(y,z,u)\neq(y',z',u)}\frac{|v(y,z,u)-v(y',z',u)|}{(|y-y'|+|z-z'|)^\alpha}.
\]
Define $F_{\alpha,k}(X^R)$ to consist of functions
with $\|v\|_{\alpha,k}=\sum_{j=0}^k \|\partial_t^jv\|_\alpha<\infty$ where $\partial_t$ denotes differentiation along the flow direction.

We can now state the main result.
Given $v\in L^1(X^R)$, $w\in L^\infty(X^R)$, define the correlation function
\[
\rho_{v,w}(t)=\int v\,w\circ f_t\,d\mu_X^R
-\int v\,d\mu_X^R \int w\,d\mu_X^R.
\]

\begin{thm} \label{thm-flow}
Assume that $f_t:X\to X$ is a $C^{1+\alpha}$ hyperbolic skew product flow
satisfying the UNI condition.
Then there exist constants $c,C>0$ such that
\[
|\rho_{v,w}(t)|\le C e^{-c t}\|v\|_{\alpha,2}\|w\|_\alpha,
\]
for all $v\in F_{\alpha,2}(X^R)$, $w\in F_\alpha(X^R)$, $t>0$.
\end{thm}

\begin{rmk}  Again, it follows by interpolation (as in Corollary~\ref{cor-main})
that if the suspension flow is an extension of an ergodic flow on a compact Riemannian manifold $M$, then exponential decay of correlations holds for H\"older
observables $v$, $w:M\to\R$.
\end{rmk}

In the remainder of this section, we prove Theorem~\ref{thm-flow}
following~\cite{AGY06}.
Let $\gamma=\gamma_0^\alpha$.
Define $\psi_t:Y^R\to\Z^+$ to be the number of visits to $Y$ by time $t$:
\[
\psi_t(y,u)=\max\{n\ge0:u+t>R_n(y)\}.
\]

\begin{prop} \label{prop-visit}
There exist $\delta$, $C>0$ such that $\int_{Y^R} \gamma^{\psi_t}\,d\mu^R\le Ce^{-\delta t}$ for all $t>0$.
\end{prop}

\begin{proof}
Write
\begin{align*}
\int_{Y^R} \gamma^{\psi_t}\,d\mu^R
& = \sum_{n=0}^\infty \gamma^n \mu^R\{(y,u):R_n(y)<u+t\le R_{n+1}(y)\}
\\ & \le \sum_{n=0}^\infty \gamma^n \mu^R\{(y,u):t\le R_{n+1}(y)\}
= (1/\bar R)\sum_{n=0}^\infty  \gamma^n
\int_Y R\, 1_{\{R_{n+1}>t\}}\,d\mu.
\end{align*}
By Cauchy-Schwarz and Markov's inequality,
\begin{align*}
\int_Y R\, 1_{\{R_n>t\}}\,d\mu
& \le |R|_2\, \mu(R_n>t)^{1/2}
= |R|_2\, \mu(e^{\delta R_n}>e^{\delta t})^{1/2}
\\ & \le |R|_2\, e^{-\frac12\delta t}\Bigl(\int_Y e^{\delta R_n}\,d\mu\Bigl)^{1/2}.
\end{align*}

Recall that the normalised twisted transfer operator $L_\sigma$ is defined for
$\sigma\in\R$ near~$0$ with leading eigenvalue $\lambda_\sigma$ satisfying $\lambda_0=1$.  We have
\[
 \int_Y e^{\delta R_n}\,d\mu 
 =\int_Y L_0^ne^{\delta R_n}\,d\mu 
 =\int_Y L_{-\delta}^n1\,d\mu,
\]
so
 $\int_Y e^{\delta R_n}\,d\mu \le \lambda_{-\delta}^n$.
 Since $\lambda_0=1$, we can shrink $\delta$ so that $\tilde\gamma=\gamma\lambda_{-\delta}^{1/2}<1$.
Then
\begin{align*}
\int_{Y^R} \gamma^{\psi_t}\,d\mu^R\ll 
\sum_{n=0}^\infty \gamma^n e^{-\frac12\delta t} \lambda_{-\delta}^{n/2}=
\sum_{n=0}^\infty \tilde\gamma^n e^{-\frac12\delta t} 
\ll e^{-\frac12\delta t},
\end{align*}
as required.
\end{proof}

\begin{pfof}{Theorem~\ref{thm-flow}}
Without loss, we may suppose that $\int_{X^R} v\,d\mu_X^R=0$.
Define the semiconjugacy $\pi^R:X^R\to Y^R$, $\pi^R(x,u)=(\pi x,u)$, so
$F_t\circ \pi^R=\pi^R\circ f_t$ and $\pi^R_*\mu_X^R=\mu^R$.
Define $w_t:Y^R\to\R$ by setting 
\[
w_t(y,u)=\int_{x\in\pi^{-1}(y)}w\circ f_t(x,u)\,d\eta_y(x).
\]
Then $\rho_{v,w}(2t)=I_1(t)+I_2(t)$, where
\begin{align*}
I_1(t) & = \int_{X^R} v\,w\circ f_{2t}\,d\mu_X^R-
 \int_{X^R} v\,w_t\circ F_t\circ\pi^R\,d\mu_X^R, \\
 I_2(t) & = \int_{X^R} v\,w_t\circ F_t\circ\pi^R\,d\mu_X^R.
\end{align*}

Now $I_1(t)= \int_{X^R} v\,(w\circ f_t- w_t\circ \pi^R)\circ f_t\,d\mu_X^R$, so
$|I_1(t)|\le |v|_\infty \int_{X^R} |w\circ f_t- w_t\circ \pi^R|\,d\mu_X^R$.
Using the definitions of $\pi^R$ and $w_t$,
\begin{align*}
w\circ f_t(x,u)- w_t\circ \pi^R(x,u) & =
w\circ f_t(x,u)- w_t(\pi x,u)\\  & =
\int_{x'\in\pi^{-1}(\pi x)}
(w\circ f_t(x,u)- w\circ f_t (x',u))\,d\eta_{\pi(x)}(x').
\end{align*}
Recall that $\gamma=\gamma_0^\alpha$. It follows from condition (v) that
\begin{align*}
|w\circ f_t(x,u)- w_t\circ \pi^R(x,u)| & \ll 
\int_{x'\in\pi^{-1}(\pi x)}|w|_\alpha \gamma^{\psi_t(\pi x,u)}d\eta_{\pi(x)}(x')
\\ & =|w|_\alpha \gamma^{\psi_t(\pi x,u)}
=|w|_\alpha \gamma^{\psi_t}\circ\pi^R(x,u).
\end{align*}
Hence 
\begin{align*}
|I_1(t)| & 
\ll |v|_\infty|w|_\alpha\int_{X^R} \gamma^{\psi_t}\circ\pi^R\,d\mu_X^R
 =|v|_\infty|w|_\alpha\int_{Y^R} \gamma^{\psi_t}\,d\mu^R.
\end{align*}
By Proposition~\ref{prop-visit}, $|I_1(t)|\ll |v|_\infty|w|_\alpha e^{-\delta t}$ for some $\delta>0$.

Next, define $\bar v:Y^R\to \R$ by setting
$\bar v(y,u)=\int_{x\in\pi^{-1}(y)} v(x,u)\,d\eta_y(x)$.
Since $\int_{X^R}v\,d\mu_X^R=0$, it follows from Proposition~\ref{prop-disint}(a) that
$\int_{Y^R}\bar v\,d\mu^R=0$.  Moreover,
$I_2(t)=\int_{Y^R}\bar v\,w_t\circ F_t\,d\mu^R
=\bar\rho_{\bar v,w_t}(t)$ where
$\bar\rho$ denotes the correlation function on $Y^R$.
By Proposition~\ref{prop-disint}(b), $\bar v\in F_{\alpha,2}(Y^R)$ and $\|\bar v\|_{\alpha,2}\ll \|v\|_{\alpha,2}$.
Clearly, $|w_t|_\infty\le |w|_\infty$.
Hence it follows from Theorem~\ref{thm-semiflow} that there exists $c>0$ such that
$|I_2(t)|\ll e^{-ct}\|\bar v\|_{\alpha,2}|w_t|_\infty
\ll  e^{-ct}\|v\|_{\alpha,2}|w|_\infty$ completing the proof.
\end{pfof}

\section{Applications to smooth flows}
\label{sec-app}

In this section, we mention applications of our results to certain open sets of Lorenz flows and Axiom~A flows in $\R^3$.

\subsection{Lorenz-like flows}

We consider $C^{\infty}$ vector fields $\fX:\R^3\to\R^3$ possessing an
equilibrium $p$ which is \emph{Lorenz-like}:
the eigenvalues of $(D\fX)_p$ are real and satisfy
\begin{align}\label{eq:Lorenz-like-equil}
  \lambda_{ss} < \lambda_s < 0 < -\lambda_s < \lambda_u.
\end{align}

The definition of geometric Lorenz attractor is fairly standard, and implies in particular that there is a robust topologically transitive attractor containing
the equilibrium $p$.  By Morales~{\em et al.}~\cite{MoralesPacificoPujals04}, such an attractor is {\em singular hyperbolic} with a dominated splitting into a one-dimensional uniformly contracting subbundle and a a two-dimensional subbundle with uniform expansion of area.  It follows that there is a uniformly contracting strong stable foliation $\cF^{ss}$ for the flow and a uniformly contracting stable foliation $\cW_g^s$ for the associated Poincar\'e map $g$.
(We refer to~\cite{MoralesPacificoPujals04} for a precise statement of these properties.  See also~\cite{AraujoPacifico}.)
Tucker~\cite{Tucker02} showed that the classical Lorenz attractor is an example of a geometric Lorenz attractor.

Quotienting $g$ along stable leaves in $\cW_g^s$ leads to a one-dimensional
map $\bar g$.
Tucker~\cite{Tucker02} proved moreover that for the classical Lorenz equation $\fX_0$, and nearby vector fields, the one dimensional map $\bar g$
is {\em locally eventually onto (l.e.o.)}.  
For convenience, we say that $\fX$ satisfies l.e.o.\ if $\bar g$ satisfies l.e.o.

It is often the case that a smoothness assumption is imposed on the foliation
$\cW_g^s$.  Here we require smoothness also of $\cF^{ss}$.
Following~\cite{AMV15}, we say that $\fX$ is {\em strongly dissipative} if
  the divergence of the vector field $\fX$ is strictly
  negative,
and moreover the eigenvalues of the singularity at $p$ satisfy the
additional constraint $\lambda_u+\lambda_{ss}<\lambda_s$.
By~\cite[Lemma~2.2]{AMV15} the foliation $\cF^{ss}$ (and hence the foliation 
$\cW_g^s$) is $C^{1+\alpha}$ for a strongly dissipative geometric Lorenz attractor.

\begin{rmk} \label{rmk-diss}
Strong dissipativity is clearly a $C^1$ open condition.
Moreover, for the classical Lorenz equations with vector field $\fX_0:\R^3\to\R^3$, we have
\[
\diver \fX_0\equiv -\textstyle{\frac{41}{3}}, \quad
\lambda_s=-\textstyle{\frac83}, \quad \lambda_u\approx 11.83, \quad \lambda_{ss}\approx -22.83, 
\]
so condition~\eqref{eq:Lorenz-like-equil} and strong dissipativity are
satisfied.  Consequently the foliations $\cF^{ss}$ and $\cW^s_g$ are $C^{1+\alpha}$ for $\fX_0$ and for nearby vector fields.
\end{rmk}

For $\alpha>0$, let $\cU_{1+\alpha}$ denote the set of $C^\infty$ vector fields
$\fX:\R^3\to\R^3$ possessing a geometric Lorenz attractor, satisfying the l.e.o.\ condition, with 
a $C^{1+\alpha}$ strong stable foliation $\cF^{ss}$ for the flow and hence
a $C^{1+\alpha}$ stable foliation $\cW^s_g$ for the Poincar\'e map.

For vector fields in $\cU_{1+\alpha}$, the quotient one-dimensional map $\bar g$ is a $C^{1+\alpha}$ nonuniformly expanding map with a ``Lorenz-like'' singularity 
corresponding to the Lorenz-like equilibrium.

Since $\bar g$ is l.e.o., we can
choose an interval $Y$ in the domain of $\bar g$ and an inducing time $\tau:Y\to\Z^+$ such that $F=\bar g^\tau:Y\to Y$ is a piecewise $C^{1+\alpha}$ uniformly expanding map
satisfying the assumptions in Section~\ref{sec-semiflow}.  
In particular, conditions~(i) and~(ii) are satisfied.  The absolutely continuous invariant probability measure for $F$ leads in a standard way to an SRB measure
$\mu$ supported on the attractor for the flow.
All mixing properties discussed below are with respect to $\mu$.

Since the strong stable foliation $\cF^{ss}$ is $C^{1+\alpha}$, it is possible as
in~\cite{ABV,AMV15,AraujoVarandas12} to choose a $C^{1+\alpha}$-embedded Poincar\'e section consisting of strong stable leaves.   The properties above
of $F=\bar g^\tau$ are unchanged, but the return time function $r$ to the Poincar\'e section is $C^{1+\alpha}$ and constant along stable leaves in $W^s_g$
and hence restricts to a $C^{1+\alpha}$ return time function, also denoted $r$, for the quotient system.
Inducing leads to a piecewise $C^{1+\alpha}$ induced return time function
$R:Y\to\R^+$ given by $R(y)=\sum_{j=0}^{\tau(y)-1}r(\bar g^jy)$.
Conditions~(iii) and~(iv) in Section~\ref{sec-semiflow} are verified
in~\cite[Section~4.2.2]{AraujoVarandas12} 
and~\cite[Section~4.2.1]{AraujoVarandas12} respectively.
Hence the quotient induced semiflow is a $C^{1+\alpha}$ expanding semiflow
as defined in Section~\ref{sec-semiflow}.
Similarly, 
the induced flow (starting with $g$ instead of $\bar g$ and using the same inducing time $\tau$) is a $C^{1+\alpha}$ hyperbolic skew product flow
as defined in Section~\ref{sec-flow}.
To establish exponential decay of correlations, it remains to verify the UNI condition.

\paragraph{Verification of UNI on $\cU_{1+\alpha}$, $\alpha>0$.}
	Ara\'ujo~{\em et al.}~\cite{AMV15} established joint nonintegrability of the stable and unstable foliations relative to a specific choice of inducing scheme (the ``double inducing scheme'' defined in~\cite[Section~3.1]{AMV15}). It is well-known~\cite{Dolgopyat98a,ABV} that joint nonintegrability is equivalent to UNI when the stable foliation is smooth.  For completeness we sketch the proof of the UNI condition.  Throughout, we assume that the inducing scheme is the one 
in~\cite[Section~3.1]{AMV15}.

Let $\alpha_0$ denote the partition $\{(c_m,d_m)\}$ of $Y$ in the definition of the uniformly expanding map $F:Y\to Y$.  Let $\alpha_1$ denote the corresponding partition of the cross-section $X$ for the Poincar\'e map $f$ for the induced flow, obtained from $\alpha_0$ by including stable leaves.
(In~\cite{AMV15}, the partition for $f$ is denoted $\alpha$. We use $\alpha_1$ here to avoid conflict with $\alpha>0$.)

Let $a\in\alpha_1$ be a partition element for $f$.
Following~\cite{AMV15}, the temporal distortion function $D:a\times a\to\R$ is defined almost everywhere on $a\times a$ by the formula
\begin{align*}
	D(y,z) & =\sum_{j=-\infty}^\infty \{r(g^jy)-r(g^j[y,z])-r(g^j[z,y])+r(g^jz)\} \\ &
=\sum_{j=-1}^\infty \{r(g^jy)-r(g^j[y,z])-r(g^j[z,y])+r(g^jz)\},
\end{align*}
where $[y,z]$ is the local product of $y$ and $z$, and the second equality follows since $r$ is constant along stable manifolds.
(Note that $f$, $F$, $\bar F$ in~\cite{AMV15} correspond to $g$, $f$, $F$ here.)
The main technical result of~\cite{AMV15} states that the stable and unstable foliations for the flow are not jointly integrable:

\begin{lemma} \label{lem-JNI}
There exists $a\in\alpha_1$ and $y,z\in a$ such that $D(y,z)\neq0$.
\end{lemma}

\begin{proof}
This is implicit in~\cite[Theorem~3.4]{AMV15}, as we now show.  

Let $\hat f:\Delta\to\Delta$ be the Young tower constructed in~\cite[Section~3.2]{AMV15} with partition $\hat\alpha$.  
(We refer to~\cite{AMV15} for the prerequisite definitions.)
Still following~\cite[Section~3.2]{AMV15}, the temporal distortion function $D$ on $Y$ extends to $\Delta$ 
via the formula
\begin{align*}
D(p,q)= \sum_{j=-\infty}^\infty \{\hat r(\hat f^jp)-\hat r(\hat f^j[p,q])-\hat r(\hat f^j[q,p])+\hat r(\hat f^jq)\} .
\end{align*}
By~\cite[Theorem~3.4]{AMV15}, there exists $\ell\ge0$ and $\hat a,\hat a'\in\hat\alpha$ lying in the $\ell$'th level of $\Delta$, and there exists $p=(y,\ell)\in\hat a$, $q=(z,\ell)\in\hat a'$, such that $D(p,q)\neq0$.

Since the local product on the tower is given by $[p,q]=([y,z],\ell)$ it follows from the definitions (recalling again that $f$ in~\cite{AMV15} is denoted $g$ here) that
\begin{align*}
D((y,\ell), & (z,\ell))  =
\sum_{j=-\infty}^\infty \{r\circ g^j(g^\ell y)-r\circ g^j(g^\ell[y,z])-r\circ g^j(g^\ell[z,y])+r\circ g^j(g^\ell z)\} 
\\ & = \sum_{j=-\infty}^\infty \{r\circ g^j(y)-r\circ g^j([y,z])-r\circ g^j([z,y])+r\circ g^j(z)\} =D(y,z).
\end{align*}
Hence $D(y,z)=D(p,q)\neq0$ as required.
\end{proof}

\begin{cor} \label{cor-UNI}
	The UNI condition holds for $\fX\in\cU_{1+\alpha}$, for all $\alpha>0$.
\end{cor}

\begin{proof}
Let $a\in\alpha_1$, $y,z\in a$.  Write $D(y,z)=D_0(y,[y,z])+D_0(z,[z,y])$ where 
\[
	D_0(y,z)=\sum_{j=1}^\infty \{r(g^{-j}y)-r(g^{-j}z)\}
\]
is defined for $y,z\in a$ lying in the same unstable manifold for $f$.
The proof of~\cite[Lemma~3.1]{AMV15} establishes that there is a sequence
of partition elements $a_i\in\alpha_1$ and
pairs of points $y_i$, $z_i\in a_i$ with
$y_0=y$, $z_0=z$ and $y_{i-1}=fy_i$, $z_{i-1}=fz_i$ for $i\ge1$, such that
\[
	D_0(y,z)=\sum_{i=1}^\infty \{R(y_i)-R(z_i)\}.
\]

Now suppose for contradiction that UNI fails.  By~\cite[Proposition~7.4]{AGY06}, there is a $C^1$ function $\zeta$ and a locally constant function $\ell$ (constant on elements of $\alpha_1$) such that $R=\zeta\circ f-\zeta+\ell$.  
Since $R$ is constant along stable manifolds and
$\zeta=-\sum_{j=0}^{n-1}R\circ f^j
+\sum_{j=0}^{n-1}\ell\circ f^j+\zeta\circ f^n$, it follows that $\zeta$
is constant along stable manifolds.
Substituting into the formulas for $D_0$ and $D$, we obtain
$D_0(y,z)=\zeta(y)-\zeta(z)$ and
\[
	D(y,z)=\zeta(y)-\zeta([y,z])+\zeta(z)-\zeta([z,y])=0.
\]
Hence $D\equiv0$ on $a\times a$ for all $a\in\alpha$,
contradicting Lemma~\ref{lem-JNI}.
\end{proof}

\begin{pfof}{Theorem~\ref{thm-Lorenz}}
It follows from~\cite{Tucker02} and Remark~\ref{rmk-diss} that there exists
$\alpha>0$ such that $\fX_0\in\Int \cU_{1+\alpha}$.
	By Theorem~\ref{thm-flow} and Corollary~\ref{cor-UNI}, exponential decay of correlations holds for all $\fX\in\cU_{1+\alpha}$.
\end{pfof}

\begin{rmk}
Previous results in the literature on mixing for Lorenz attractors
are as follows.
(For simplicity, we do not state the optimal conditions under which each individual result is known to be valid.)

By~\cite{Ratner78}, weak mixing implies Bernoulli (and hence mixing) for vector fields in $\cU_{1+\alpha}$, and~\cite{LMP05} showed that these properties indeed hold.
Moreover, by~\cite{AMV15}, all vector fields in $\cU_{1+\alpha}$ have superpolynomial decay of correlations.  That is, for $C^\infty$ observables $v,\,w:\R^3\to\R$ the correlation function $\rho_{v,w}(t)$ decays faster than any polynomial rate.   These results apply to the classical Lorenz equations $\fX_0$.

The first results on exponential decay of correlations for geometric Lorenz attractors were obtained by~\cite{AraujoVarandas12} who showed that vector fields in $\cU_2$ have exponential decay of correlations (for H\"older observables).  The above remarks show that this set has nonempty interior.
However, it seems unlikely that $\fX_0\in \cU_2$, so the classical Lorenz attractor was not covered.

Finally, we note related work of~\cite{ButterleyEslamiapp}
and~\cite{Eslamiapp} which sets out a program to prove exponential decay of correlations for maps and flows with discontinuities.
\end{rmk}

\begin{cor}  \label{cor-Lorenz} For vector fields in $\cU_{1+\alpha}$,  the ASIP for the time-$1$ map
of the corresponding flow
(cf.\ Corollary~\ref{cor-time1}) holds for H\"older mean zero observables $v:\R^3\to\R$.
\end{cor}

\begin{proof}  
This follows from Theorem~\ref{thm-Lorenz} by the methods in~\cite{AMV15}.
\end{proof}

\begin{rmk}  Theorem~C in~\cite{AMV15} already covers the ASIP for the time-$1$ map if $\fX\in\cU_{1+\alpha}$, provided the observable $v$ is $C^\infty$.
The result here applies to all H\"older observables.
\end{rmk}

\begin{rmk}  The results presented in this subsection rely heavily on~\cite{AMV15}.
The only parts of~\cite{AMV15} that are redundant are Subsections~3.4 and 3.5 together with Proposition~2.6 and the last statement of Proposition 2.5.  
\end{rmk}

\subsection{Axiom~A flows}

In~\cite{ABV}, it is shown that in all dimensions greater than two,
there is an open set of Axiom~A flows with exponential decay of correlations.
Roughly speaking, these are flows with $C^2$ strong stable foliation
(forced by a domination condition which is robust) satisfying the UNI condition.  

An immediate consequence of Theorem~\ref{thm-flow} is that we recover and extend the result in~\cite{ABV} in the three-dimensional case, since
we require only that the strong stable foliation is $C^{1+\alpha}$ (which is forced by a weaker domination condition).  
We conjecture that the same is true in higher dimensions.  To prove this it would be necessary to check that our extension of~\cite{BaladiVallee05} to the
$C^{1+\alpha}$ situation works in the higher-dimensional setting
of~\cite[Section~2.1]{AGY06}.
We have chosen to restrict attention to the lowest dimensional situation in this paper since it avoids certain technicalities and it suffices for the important case of Lorenz attractors.

\appendix

\section{Correlation function of a suspension semiflow}
\label{sec-Poll}

In this appendix, we recall a formula of Pollicott~\cite{Pollicott85} for the 
correlation function corresponding to a suspension semiflow or flow.

\begin{prop}  \label{prop-J}
$\rho_{v,w}(t)=\sum_{n=0}^\infty J_n(t)$ where
\[
J_n(t)=\begin{cases} \int_{Y^R} 1_{\{t+u<R(y)\}}v(y,u)w(y,t+u)\,d\mu^R, & n=0 \\
\int_{Y^R} 1_{\{R_n(y)<t+u<R_{n+1}(y)\}}v(y,u)w(F^ny,t+u-R_n(y))\,d\mu^R, & n\ge1\end{cases}
\]
\end{prop}

\begin{proof}
Write
\begin{align*}
\rho(t) & =\int_{Y^R} 1_{\{u<t+u\}}v(y,u)\,w\circ F_t(y,u)\,d\mu^R
\\ & =\int_{Y^R} 1_{\{u<t+u<R(y)\}}v(y,u)\,w\circ F_t(y,u)\,d\mu^R
\\ & \qquad \qquad +\sum_{n=1}^\infty\int_{Y^R} 1_{\{R_n(y)<t+u<R_{n+1}(y)\}}v(y,u)\,w\circ F_t(y,u)\,d\mu^R.
\end{align*}
The result follows.
\end{proof}

Let $\hat \rho_{v,w}(s)$ denote the Laplace transform of $\rho_{v,w}(t)$.
Similarly, let $\hat J_n(s)$ denote the Laplace transform of $J_n(t)$.
We note that $\hat\rho$ and $\hat J_n$ are analytic on
$\{s\in\C:\Re s>0\}$.
It is easily checked that $\hat J_n$ is analytic on the whole of $\C$ for each $n$.  We require explicit bounds for the case $n=0$.

\begin{prop} \label{prop-J0}
For each $a>0$ there exists $C>0$ such that
$|J_0(s)|\le C|v|_\infty|w|_\infty$ for all $s\in\C$ with $\Re s\ge -a$.
\end{prop}

\begin{proof}
Begin by writing
\begin{align*}
\hat J_0(s)
& =\int_0^\infty e^{-st}\int_{Y^R}1_{\{t+u<R(y)\}}v(y,u)w(y,t+u)\,d\mu^R\,dt \\ &
=(1/\bar R)\int_Y\int_0^{R(y)}\int_0^{R(y)-u} e^{-st}v(y,u)w(y,t+u)\,dt\,du\,d\mu.
\end{align*}
Hence
\begin{align*}
|\hat J_0(s)| & \le |v|_\infty|w|_\infty (1/\bar R) \int_Y\int_0^{R(y)}\int_0^{R(y)-u} e^{at}\,dt\,du\,d\mu
\\ & \le |v|_\infty|w|_\infty (1/\bar R)a^{-2}\int_Y e^{a R}\,d\mu
\ll |v|_\infty|w|_\infty 
\end{align*}
by Proposition~\ref{prop-ReR}.
\end{proof}

\begin{prop} \label{prop-hatJ}
For $n\ge1$, $\hat J_n(s)=(\bar R)^{-1}\int_Y e^{-sR_n}v_s\,w_s\circ F^n\,d\mu$ where
\[
v_s(y)=\int_0^{R(y)}e^{su}v(y,u)\,du,
\quad w_s(y)=\int_0^{R(y)}e^{-su}w(y,u)\,du.
\]
\end{prop}

\begin{proof}
First note that
\begin{align*}
\hat J_n(s) & =\int_0^\infty e^{-st}J_n(t)\,dt
\\ & =(\bar R)^{-1}\int_Y\int_0^{R(y)}\int_{R_n(y)-u}^{R_{n+1}(y)-u}e^{-st}v(y,u)w(F^ny,t+u-R_n(y))\,dt\,du\,d\mu.
\end{align*}
The substitution $u'=t+u-R_n(y)$ yields
\begin{align*}
\hat J_n(s) & =(\bar R)^{-1}\int_Y\Bigl(\int_0^{R(y)}e^{su}v(y,u)\,du\Bigr)
\Bigl(\int_0^{R(F^ny)}e^{-su'}w(F^ny,u')\,du'\Bigr)e^{-sR_n(y)}\,d\mu,
\end{align*}
which is the required formula.
\end{proof}

\paragraph{Acknowledgement} 
V.A.  was partially supported by CNPq, PRONEX-Dyn.Syst. and FAPESB (Brazil).  
I.M. was partially supported by a European
  Advanced Grant StochExtHomog (ERC AdG 320977) and by CNPq
  (Brazil) through PVE grant number 313759/2014-6.

\end{document}